\documentclass[11pt,a4paper]{article}

\usepackage{inputenc}
\usepackage{amsmath}
\usepackage{bm}
\usepackage{bbold}
\usepackage{amsthm}
\usepackage{enumerate}

\usepackage{hyperref}

\title{Algebraic solution of tropical optimization problems via matrix sparsification with application to scheduling}

\author{N. Krivulin\thanks{Faculty of Mathematics and Mechanics, Saint Petersburg State University, 28 Universitetsky Ave., St.~Petersburg, 198504, Russia, 
nkk@math.spbu.ru.}
}

\date{}

\newtheorem{theorem}{Theorem}
\newtheorem{lemma}[theorem]{Lemma}
\newtheorem{corollary}[theorem]{Corollary}

\theoremstyle{definition}
\newtheorem{example}{Example}

\begin{document}

\maketitle

\begin{abstract}
Optimization problems are considered in the framework of tropical algebra to minimize and maximize a nonlinear objective function defined on vectors over an idempotent semifield, and calculated using multiplicative conjugate transposition. To find the minimum of the function, we first obtain a partial solution, which explicitly represents a subset of solution vectors. We characterize all solutions by a system of simultaneous equation and inequality, and show that the solution set is closed under vector addition and scalar multiplication. A matrix sparsification technique is proposed to extend the partial solution, and then to obtain a complete solution described as a family of subsets. We offer a backtracking procedure that generates all members of the family, and derive an explicit representation for the complete solution. As another result, we deduce a complete solution of the maximization problem, given in a compact vector form by the use of sparsified matrices. The results obtained are illustrated with illuminating examples and graphical representations. We apply the results to solve real-world problems drawn from project (machine) scheduling, and give numerical examples.
\\

\\

\end{abstract}

\section{Introduction}

Tropical (idempotent) mathematics focuses on the theory and applications of semirings with idempotent addition, and had its origin in the seminal works published in the 1960s by Pandit \cite{Pandit1961Anew}, Cuninghame-Green \cite{Cuninghamegreen1962Describing}, Giffler \cite{Giffler1963Scheduling}, Hoffman \cite{Hoffman1963Onabstract}, Vorob{'}ev \cite{Vorobjev1963Theextremal}, Romanovski{\u\i} \cite{Romanovskii1964Asymptotic}, Korbut \cite{Korbut1965Extremal}, and Peteanu \cite{Peteanu1967Analgebra}. An extensive study of tropical mathematics was motivated by real-world problems in various areas of operations research and computer science, including path analysis in graphs and networks \cite{Pandit1961Anew,Peteanu1967Analgebra}, machine scheduling \cite{Cuninghamegreen1962Describing,Giffler1963Scheduling}, production planning and control \cite{Romanovskii1964Asymptotic,Vorobjev1963Theextremal}. The significant progress in the field over the past few decades is reported in several research monographs, such as ones by Baccelli at al. \cite{Baccelli1993Synchronization}, Kolokoltsov and Maslov \cite{Kolokoltsov1997Idempotent}, Golan \cite{Golan2003Semirings}, Heidergott et al. \cite{Heidergott2006Maxplus}, McEneaney \cite{Mceneaney2006Maxplus}, Gondran and Minoux \cite{Gondran2008Graphs}, Butkovi\v{c} \cite{Butkovic2010Maxlinear}, Maclagan and Sturmfels \cite{Maclagan2015Introduction} as well as in a wide range of contributed papers.

Since the early studies \cite{Giffler1963Scheduling,Hoffman1963Onabstract,Peteanu1967Analgebra,Romanovskii1964Asymptotic}, optimization problems that can be examined in the framework of tropical mathematics have formed a notable research domain in the field. These problems are formulated to minimize or maximize functions defined on vectors over idempotent semifields (semirings with multiplicative inverses), and may involve constraints in the form of vector equations and inequalities. The objective functions can be both linear and nonlinear in the tropical mathematics setting. 

The span (range) vector seminorm, which is defined as the maximum deviation between components of a vector, presents one of the objective functions encountered in practice. Specifically, this seminorm can serve as the optimization criterion for just-in-time scheduling (see, e.g., Demeulemeester and Herroelen \cite{Demeulemeester2002Project}, Neumann et al. \cite{Neumann2003Project}, T{\textquoteright}kindt and Billaut \cite{Tkindt2006Multicriteria} and Vanhoucke \cite{Vanhoucke2012Project}), and finds applications in real-world problems that involve time synchronization in manufacturing, transportation networks, and parallel data processing.

In the context of tropical mathematics, the span seminorm has been introduced by Cuninghame-Green \cite{Cuninghamegreen1979Minimax}, and Cuninghame-Green and Butkovi\v{c} \cite{CuninghameGreen2004Bases}. The seminorm was used by Butkovi\v{c} and Tam \cite{Butkovic2009Onsome} and Tam \cite{Tam2010Optimizing} in optimization problems drawn from machine scheduling. A manufacturing system was considered, in which machines start and finish under some precedence constraints to make components for final products. The problems were to find the starting time for each machine to provide the completion times that are spread over either the shortest or longest time intervals. Solutions were given within a combined framework that involves two reciprocally dual idempotent semifields. Similar problems in the general setting of tropical mathematics were examined by Krivulin in \cite{Krivulin2013Explicit,Krivulin2016Amaximization,Krivulin2017Tropical}, where direct, explicit solutions were suggested. However, the results obtained present a partial solution, rather than a complete solution, or offer a solution in scalar terms, rather than in a compact vector form.

We consider the tropical optimization problems formulated in \cite{Krivulin2013Explicit,Krivulin2016Amaximization,Krivulin2017Tropical} as extensions of the problems of minimizing and maximizing the span seminorm, and represent them in a slightly different form to 
\begin{equation*}
\begin{aligned}
&
\text{minimize (maximize)}
&&
\bm{q}^{-}\bm{x}(\bm{A}\bm{x})^{-}\bm{p},
\end{aligned}
\end{equation*}
where $\bm{p}$ and $\bm{q}$ are given vectors, $\bm{A}$ is a given matrix, $\bm{x}$ is the unknown vector. The minus sign in the superscript indicates multiplicative conjugate transposition of vectors, and the matrix-vector multiplications are thought of in the sense of tropical algebra.

The purpose of this paper is twofold. First, to obtain complete solutions to both minimization and maximization problems in an explicit vector form. We extend the partial solution of the minimization problem, which is obtained in \cite{Krivulin2013Explicit} in the form of a subset of solution vectors, to a complete solution, describing all vectors that solve the problem. We combine the approach developed in \cite{Krivulin2012Asolution,Krivulin2013Explicit,Krivulin2015Extremal,Krivulin2015Amultidimensional,Krivulin2006Solution} to reduce the problem to a system of simultaneous equation and inequality, with a new matrix sparsification technique to describe all solutions to the system. We use sparsified matrices to transform the complete solution of the maximization problem given in \cite{Krivulin2016Amaximization} into a compact vector form as well. 

The second purpose is to apply the above results to the solution of real-world problems taken from just-in-time and scarce resource scheduling. We consider a project that involves a set of activities operating in parallel under temporal constraints imposed on the start and finish times of activities in the form of start-start, start-finish and finish-start precedence relations, and the finish deadline time boundaries. The problems are to minimize or maximize the maximum deviation of the finish times of activities, subject to the given constraints. These scheduling objectives reflect various possible resource limitations, such as manpower, energy and location constraints, which can require that all activities be finished simultaneously, or conversely, that the finish times be spread over the longest time interval.

We use results in \cite{Krivulin2013Explicit,Krivulin2014Aconstrained,Krivulin2015Extremal,Krivulin2015Amultidimensional}, which enable to represent a range of scheduling problems as optimization problems in terms of tropical algebra, and then to obtain direct closed-form solutions to the problems on the basis of methods of tropical optimization. Note that existing solutions to the problems of interest generally present iterative algorithms that produce a solution if any exists, or indicate that there are no solutions (see, e.g., \cite{Demeulemeester2002Project,Neumann2003Project,Tkindt2006Multicriteria,Vanhoucke2012Project} for further details and comprehensive reviews). Moreover, many problems can be expressed as linear and mixed-integer linear programs, and then solved using an appropriate computational scheme of (mixed-) linear programming, which, in general, does not guarantee a direct solution in a closed form. 

This paper further extends and supplements the results presented in the conference paper \cite{Krivulin2015Soving}, which examined only the minimization problem, and focused on theoretical aspects of tropical optimization, rather than on applications of the results. Specifically, the current paper offers a new complete solution, obtained in a compact vector form by the use of sparsified matrices, to the maximization problem under study as well. In addition to the theoretical results, the paper describes, in detail, the application of the results to solve scheduling problems, and gives illuminating examples.

The solutions obtained for the scheduling problems under both minimization and maximization of the maximum deviation of finish times present quite new results. For instance, we derive a complete solution of the scheduling problem under the minimization criterion, which significantly extends previously known partial solutions \cite{Krivulin2013Explicit,Krivulin2017Tropical}. The maximization problem under study generalizes those considered in \cite{Krivulin2016Amaximization} by taking into consideration additional constraints. Moreover, we offer a solution to the problem, which, in contrast to the scalar representation in \cite{Krivulin2016Amaximization}, is given in a compact vector form, ready for further analysis and practical use. 

We start with a brief overview of basic definitions, notation, and preliminary results of tropical mathematics in Section~\ref{S-PR} to provide a general framework for the solutions in the later sections. Specifically, a lemma that offers two equivalent representations for a vector set is presented, which is of independent interest. Section~\ref{S-TOP} presents formulations for both minimization and maximization problems under consideration.

To solve the minimization problem in Section~\ref{S-SMinP}, we first find the minimum in the problem, and offer a partial solution in the form of an explicit representation of a subset of solution vectors. We characterize all solutions to the problem by a system of simultaneous equation and inequality, and exploit this characterization to investigate properties of the solutions. Furthermore, we develop a matrix sparsification technique, which consists in dropping entries below a prescribed threshold in the matrix of the problem without affecting the solution. By combining this technique with the above characterization, the partial solution obtained is extended to a wider solution subset, which includes the partial solution as a special case. We describe all solutions of the problem as a family of subsets, and propose a backtracking procedure that allows one to generate all members in the family. The section concludes with the main result, which offers an explicit representation for the complete solution in a compact vector form.

In Section~\ref{S-SMaxP}, we apply the above representation lemma to describe a complete solution of the maximization problem in a compact vector form using sparsified matrices. Numerical examples and graphical illustrations are included in this and previous sections to provide additional insights into the results obtained. Finally, in Section~\ref{S-ASP}, we solve real-world problems drawn from scheduling. We start with a standard description of the scheduling problems, and then represent them in terms of tropical algebra. We use the previous results to obtain complete solutions of the problems, and then give examples to illustrate the solution.

\section{Preliminary results}
\label{S-PR}

In this section, we give a brief overview of the main definitions, notation, and preliminary results used in the subsequent solution to handle the tropical optimization problems under study. Concise introductions to and thorough discussion of tropical mathematics are presented in various forms in a range of works, including \cite{Akian2007Maxplus,Baccelli1993Synchronization,Butkovic2010Maxlinear,Golan2003Semirings,Gondran2008Graphs,Heidergott2006Maxplus,Kolokoltsov1997Idempotent,Litvinov2007Themaslov,Maclagan2015Introduction,Mceneaney2006Maxplus}. In the overview below, we mainly follow the results in \cite{Krivulin2014Aconstrained,Krivulin2015Extremal,Krivulin2015Amultidimensional}, which offer a unified framework to obtain explicit solutions in a compact form. For further details, one can consult the publications listed before.

\subsection{Idempotent semifield}

Let $\mathbb{X}$ be a nonempty set that is closed under two associative and commutative operations, addition $\oplus$ and multiplication $\otimes$, which have their neutral elements, zero $\mathbb{0}$ and identity $\mathbb{1}$. Addition is idempotent to yield $x\oplus x=x$ for all $x\in\mathbb{X}$. Multiplication is invertible, which implies that each nonzero $x\in\mathbb{X}$ has an inverse $x^{-1}$ to satisfy the equality $x\otimes x^{-1}=\mathbb{1}$. Moreover, multiplication distributes over addition, and has $\mathbb{0}$ as the absorbing element. Under these conditions, the system $\langle\mathbb{X},\mathbb{0},\mathbb{1},\oplus,\otimes\rangle$ is commonly referred to as the idempotent semifield.

The idempotent addition produces a partial order, by which $x\leq y$ if and only if $x\oplus y=y$. With respect to this order, the inequality $x\oplus y\leq z$ is equivalent to two inequalities $x\leq z$ and $y\leq z$. Moreover, addition and multiplication are isotone in each argument, whereas the multiplicative inversion is antitone.

The partial order is assumed to extend to a consistent total order over $\mathbb{X}$.

The power notation with integer exponents is used for iterated multiplication to define $x^{0}=\mathbb{1}$, $x^{p}=x\otimes x^{p-1}$, $x^{-p}=(x^{-1})^{p}$ for any nonzero $x$ and positive integer $p$. In what follows, the multiplication sign $\otimes$ is dropped for simplicity. The relation symbols and the optimization problems are thought of in terms of the above order, which is induced by idempotent addition.

As examples of the general semifield under consideration, one can take
$$
\begin{aligned}
\mathbb{R}_{\max,+}
&=
\langle\mathbb{R}\cup\{-\infty\},-\infty,0,\max,+\rangle,
\\
\mathbb{R}_{\min,+}
&=
\langle\mathbb{R}\cup\{+\infty\},+\infty,0,\min,+\rangle,
\\
\mathbb{R}_{\max,\times}
&=
\langle\mathbb{R}_{+}\cup\{0\},0,1,\max,\times\rangle,
\\
\mathbb{R}_{\min,\times}
&=
\langle\mathbb{R}_{+}\cup\{+\infty\},+\infty,1,\min,\times\rangle,
\end{aligned}
$$
where $\mathbb{R}$ is the set of real numbers, and $\mathbb{R}_{+}=\{x\in\mathbb{R}|x>0\}$. 

Specifically, the semifield $\mathbb{R}_{\max,+}$ has addition $\oplus$ given by the maximum, and multiplication $\otimes$ by the ordinary addition, with the null $\mathbb{0}=-\infty$ and identity $\mathbb{1}=0$. Each $x\in\mathbb{R}$ has its inverse $x^{-1}$ equal to $-x$ in standard notation. The power $x^{y}$ is defined for any $x,y\in\mathbb{R}$ and coincides with the arithmetic product $xy$. The order induced by addition corresponds to the natural linear order on $\mathbb{R}$.

\subsection{Matrix and vector algebra}

We now consider matrices over $\mathbb{X}$ and denote the set of matrices with $m$ rows and $n$ columns by $\mathbb{X}^{m\times n}$. A matrix with all entries equal to $\mathbb{0}$ is called the zero matrix. A matrix is row- (column-) regular, if it has no zero rows (columns).

For any matrices $\bm{A}=(a_{ij})$, $\bm{B}=(b_{ij})$, and $\bm{C}=(c_{ij})$ of appropriate size, and a scalar $x$, matrix addition, matrix and scalar multiplications are routinely defined entry-wise by the formulae
$$
\{\bm{A}\oplus\bm{B}\}_{ij}
=
a_{ij}\oplus b_{ij},
\quad
\{\bm{B}\bm{C}\}_{ij}
=
\bigoplus_{k}b_{ik}c_{kj},
\quad
\{x\bm{A}\}_{ij}
=
xa_{ij}.
$$

The partial order induced on $\mathbb{X}$ by idempotent addition, and its properties are extended to matrices, where the relations are considered entry-wise.

For any matrix $\bm{A}\in\mathbb{X}^{m\times n}$, its transpose is the matrix $\bm{A}^{T}\in\mathbb{X}^{n\times m}$.

For a nonzero matrix $\bm{A}=(a_{ij})\in\mathbb{X}^{m\times n}$, the multiplicative conjugate transpose is the matrix $\bm{A}^{-}=(a_{ij}^{-})\in\mathbb{X}^{n\times m}$ with the elements $a_{ij}^{-}=a_{ji}^{-1}$ if $a_{ji}\ne\mathbb{0}$, and $a_{ij}^{-}=\mathbb{0}$ otherwise. 

Consider square matrices in the set $\mathbb{X}^{n\times n}$. A matrix that has the diagonal entries equal to $\mathbb{1}$, and the off-diagonal entries to $\mathbb{0}$, is the identity matrix denoted by $\bm{I}$. Nonnegative powers represent repeated matrix multiplication as $\bm{A}^{0}=\bm{I}$ and $\bm{A}^{p}=\bm{A}\bm{A}^{p-1}$ for any nonzero matrix $\bm{A}$ and integer $p\geq1$.

The trace of a square matrix $\bm{A}\in\mathbb{X}^{n\times n}$ is defined by
$$
\mathop\mathrm{tr}\bm{A}
=
\bigoplus_{i=1}^{n}a_{ii}.
$$ 

Suppose that a square matrix $\bm{A}$ is row-regular. Clearly, the inequality $\bm{A}\bm{A}^{-}\geq\bm{I}$ is then valid. Moreover, if the row-regular matrix $\bm{A}$ has exactly one nonzero entry in every row, then the inequality $\bm{A}^{-}\bm{A}\leq\bm{I}$ holds as well. 

The matrices with only one column (row) are routinely referred to as the column (row) vectors. Unless otherwise indicated, the vectors are considered below as column vectors. The set of column vectors of order $n$ is denoted by $\mathbb{X}^{n}$.

A vector that has all components equal to $\mathbb{0}$ is the zero vector denoted $\bm{0}$. If a vector has no zero components, it is called regular. 

For any vectors $\bm{a}=(a_{i})$ and $\bm{b}=(b_{i})$ of the same order, and a scalar $x$, addition and scalar multiplication are performed component-wise by the rules
$$
\{\bm{a}\oplus\bm{b}\}_{i}
=
a_{i}\oplus b_{i},
\quad
\{x\bm{a}\}_{i}
=
xa_{i}.
$$

In the context of $\mathbb{R}_{\max,+}^{2}$, these vector operations are illustrated in the Cartesian coordinate system on the plane in Fig.~\ref{F-ASMLSV}.
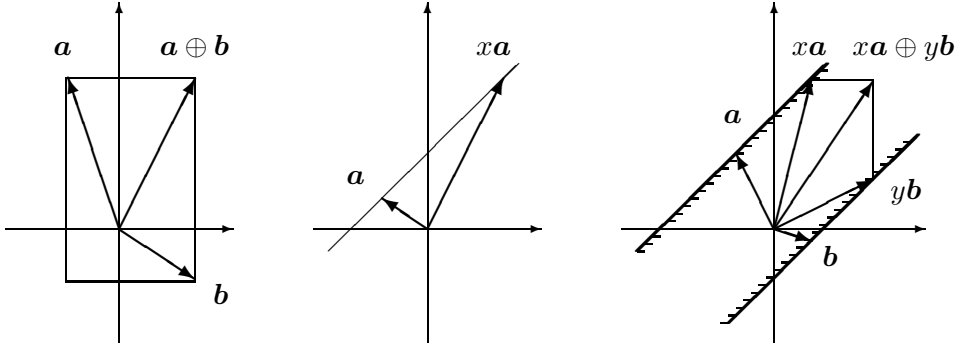
\begin{figure}[ht]
\setlength{\unitlength}{1mm}
\begin{picture}(30,45)

\put(0,15){\vector(1,0){30}}
\put(15,0){\vector(0,1){45}}


\put(15,15){\thicklines\vector(-1,3){6.7}}

\put(8,35){\line(1,0){17}}
\put(8,35){\line(0,-1){27}}


\put(15,15){\thicklines\vector(3,-2){10}}

\put(25,8){\line(-1,0){17}}
\put(25,8){\line(0,1){27}}




\put(15,15){\thicklines\vector(1,2){10}}

\put(20,38){$\bm{a}\oplus\bm{b}$}

\put(6,38){$\bm{a}$}

\put(27,5){$\bm{b}$}


\end{picture}
\hspace{8\unitlength}
\begin{picture}(30,45)

\put(0,15){\vector(1,0){30}}
\put(15,0){\vector(0,1){45}}


\put(15,15){\thicklines\vector(1,2){10}}

\put(2,12){\line(1,1){25}}

\put(15,15){\thicklines\vector(-3,2){6}}

\put(21,38){$x\bm{a}$}

\put(4,21){$\bm{a}$}


\end{picture}
\hspace{8\unitlength}
\begin{picture}(40,45)

\put(0,15){\vector(1,0){40}}
\put(20,0){\vector(0,1){45}}


\put(20,15){\thicklines\vector(-1,2){5}}

\put(1.9,12){\thicklines\line(1,1){25}}
\put(2.1,12){\thicklines\line(1,1){25}}
\multiput(2,12)(1,1){25}{\line(1,0){1}}

\put(20,15){\thicklines\vector(1,4){5}}

\put(20,15){\thicklines\vector(3,-1){5}}

\put(13.9,2.5){\thicklines\line(1,1){25}}
\put(14.1,2.5){\thicklines\line(1,1){25}}
\multiput(14,2.5)(1,1){25}{\line(-1,0){1}}

\put(20,15){\thicklines\vector(2,1){13}}

\put(20,15){\thicklines\vector(2,3){13}}

\put(33,34.75){\line(-1,0){8}}
\put(33,34.75){\line(0,-1){13}}

\put(13,29){$\bm{a}$}

\put(22,38){$x\bm{a}$}

\put(26,10){$\bm{b}$}

\put(35,19){$y\bm{b}$}

\put(30,38){$x\bm{a}\oplus y\bm{b}$}


\end{picture}
\caption{Addition (left), scalar multiplication (middle), and a linear span (right) of vectors in $\mathbb{R}_{\max,+}^{2}$.}
\label{F-ASMLSV}
\end{figure}

The left picture shows that, in terms of $\mathbb{R}_{\max,+}^{2}$, vector addition follows a rectangle rule. The sum of two vectors is the upper right vertex of the rectangle formed by the lines that are drawn through the ends of the vectors parallel to the coordinate axes. Scalar multiplication is given in the middle by the shift of the end point of a vector along the line at $45^{\circ}$ to the axes.

Let $\bm{x}$ be a regular vector and $\bm{A}$ be a square matrix of the same order. It is clear that the vector $\bm{A}\bm{x}$ is regular only when the matrix $\bm{A}$ is row-regular. Similarly, the row vector $\bm{x}^{T}\bm{A}$ is regular provided that $\bm{A}$ is column-regular. 

For any nonzero vector $\bm{x}=(x_{i})\in\mathbb{X}^{n}$, the multiplicative conjugate transpose is the row vector $\bm{x}^{-}=(x_{i}^{-})$, where $x_{i}^{-}=x_{i}^{-1}$ if $x_{i}\ne\mathbb{0}$, and $x_{i}^{-}=\mathbb{0}$ otherwise. The following properties of the conjugate transposition are easy to verify.

For any nonzero vectors $\bm{x}$ and $\bm{y}$, the equality $(\bm{x}\bm{y}^{-})^{-}=\bm{y}\bm{x}^{-}$ is valid. When the vectors $\bm{x}$ and $\bm{y}$ are regular and have the same size, the component-wise inequality $\bm{x}\leq\bm{y}$ implies that $\bm{x}^{-}\geq\bm{y}^{-}$ and vice versa. 

For any nonzero column vector $\bm{x}$, the equality $\bm{x}^{-}\bm{x}=\mathbb{1}$ holds (here and thereafter we identify $(1\times1)$-matrices with scalars). Moreover, if the vector $\bm{x}$ is regular, then the matrix inequality $\bm{x}\bm{x}^{-}\geq\bm{I}$ is valid as well.

\subsection{Linear dependence}

A vector $\bm{b}\in\mathbb{X}^{m}$ is linearly dependent on vectors $\bm{a}_{1},\ldots,\bm{a}_{n}\in\mathbb{X}^{m}$ if there exist scalars $x_{1},\ldots,x_{n}\in\mathbb{X}$ such that the vector $\bm{b}$ can be represented by a linear combination of these vectors as $\bm{b}=x_{1}\bm{a}_{1}\oplus\cdots\oplus x_{n}\bm{a}_{n}$. Specifically, the vector $\bm{b}$ is collinear with a vector $\bm{a}$  if $\bm{b}=x\bm{a}$ for some scalar $x$.

To describe a formal criterion for a vector $\bm{b}$ to be linearly dependent on vectors $\bm{a}_{1},\ldots,\bm{a}_{n}$, we take the latter vectors to form the matrix $\bm{A}=(\bm{a}_{1},\ldots,\bm{a}_{n})$, and then introduce a function that maps the pair $(\bm{A},\bm{b})$ to the scalar
$$
\delta(\bm{A},\bm{b})
=
(\bm{A}(\bm{b}^{-}\bm{A})^{-})^{-}\bm{b}.
$$
 
The following result was obtained in \cite{Krivulin2006Solution} (see also \cite{Krivulin2012Asolution}). 
\begin{lemma}
\label{E-AbAb1}
A vector $\bm{b}$ is linearly dependent on vectors $\bm{a}_{1},\ldots,\bm{a}_{n}$ if and only if the condition $\delta(\bm{A},\bm{b})=\mathbb{1}$ holds, where $\bm{A}=(\bm{a}_{1},\ldots,\bm{a}_{n})$.
\end{lemma}

The set of all linear combinations of vectors $\bm{a}_{1},\ldots,\bm{a}_{n}\in\mathbb{X}^{m}$ form a linear span of the vectors, which is closed under vector addition and scalar multiplication. A linear span of two vectors in $\mathbb{R}_{\max,+}^{2}$ is displayed in Fig.~\ref{F-ASMLSV} (right) as a strip between two thick hatched lines drawn at $45^{\circ}$ to the axes.

A system of vectors $\bm{a}_{1},\ldots,\bm{a}_{n}$ is linearly dependent if at least one vector in the system is linearly dependent on others, and independent otherwise.

Two systems of vectors are considered equivalent if each vector of one system is a linear combination of vectors of the other system. Equivalent systems of vectors obviously have a common linear span.

Let $\bm{a}_{1},\ldots,\bm{a}_{n}$ be a system that may include linearly dependent vectors. To construct an equivalent linearly independent system, we use a procedure that sequentially reduces the system until it becomes linearly independent. The procedure applies the criterion provided by Lemma~\ref{E-AbAb1} to examine the vectors one by one to remove a vector if it is linearly dependent on others, or to leave the vector in the system otherwise. It is not difficult to see that the procedure results in a linearly independent system that is equivalent to the original one.

\subsection{Solution to vector inequalities}

We start with an inequality that appears in many studies in different settings, and has solutions known in various forms. Suppose that, given a matrix $\bm{A}\in\mathbb{X}^{m\times n}$ and a regular vector $\bm{d}\in\mathbb{X}^{m}$, the problem is to find vectors $\bm{x}\in\mathbb{X}^{n}$ that solve the inequality
\begin{equation}
\bm{A}\bm{x}
\leq
\bm{d}.
\label{I-Ax-d}
\end{equation}

A direct solution proposed in \cite{Krivulin2015Extremal} is described as follows.
\begin{lemma}
\label{L-Ax-d}
For any column-regular matrix $\bm{A}$ and regular vector $\bm{d}$, all solutions to inequality \eqref{I-Ax-d} are given by
\begin{equation*}
\bm{x}
\leq
(\bm{d}^{-}\bm{A})^{-}.
\label{I-xd-A}
\end{equation*}
\end{lemma}

Next, we consider the following problem: given a matrix $\bm{A}\in\mathbb{X}^{n\times n}$, find regular vectors $\bm{x}\in\mathbb{X}^{n}$ to satisfy the inequality
\begin{equation}
\bm{A}\bm{x}
\leq
\bm{x}.
\label{I-Ax-x}
\end{equation}

To describe a solution to the problem in a compact vector form, we define a function that takes any matrix $\bm{A}\in\mathbb{X}^{n\times n}$ to the scalar
\begin{equation*}
\mathop\mathrm{Tr}(\bm{A})
=
\bigoplus_{k=1}^{n}
\mathop\mathrm{tr}\bm{A}^{k}.
\end{equation*}

Provided that the condition $\mathop\mathrm{Tr}(\bm{A})\leq\mathbb{1}$ holds, we use the asterisk operator (also known as the Kleene star), which maps $\bm{A}$ to the matrix
\begin{equation*}
\bm{A}^{\ast}
=
\bigoplus_{k=0}^{n-1}\bm{A}^{k}.
\end{equation*}

The following result obtained in \cite{Krivulin2015Amultidimensional,Krivulin2006Solution} by using various arguments offers a direct solution to inequality \eqref{I-Ax-x}.
\begin{theorem}
\label{T-Ax-x}
For any matrix $\bm{A}$, the following statements hold:
\begin{enumerate}
\item If $\mathop\mathrm{Tr}(\bm{A})\leq\mathbb{1}$, then all regular solutions to \eqref{I-Ax-x} are given by $\bm{x}=\bm{A}^{\ast}\bm{u}$, where $\bm{u}$ is any regular vector.
\item If $\mathop\mathrm{Tr}(\bm{A})>\mathbb{1}$, then there is no regular solution.
\end{enumerate}
\end{theorem}

\subsection{Representation lemma}

We apply properties of the conjugate transposition to obtain a useful result that offers an equivalent representation for a set of vectors $\bm{x}\in\mathbb{X}^{n}$, which is defined by boundaries given by a parametrized double inequality.

\begin{lemma}
\label{L-alphagxalphah}
Let $\bm{g}$ be a vector and $\bm{h}$ a regular vector such that $\bm{g}\leq\bm{h}$. Then, the following statements are equivalent:
\begin{enumerate}
\item
The vector $\bm{x}$ satisfies the double inequality
\begin{equation}
\alpha\bm{g}
\leq
\bm{x}
\leq
\alpha\bm{h},
\quad
\alpha
>
\mathbb{0}.
\label{I-alphagxalphah}
\end{equation}
\item
The vector $\bm{x}$ is given by the equality
\begin{equation}
\bm{x}
=
(\bm{I}\oplus\bm{g}\bm{h}^{-})\bm{u},
\quad
\bm{u}
>
\bm{0}.
\label{E-xIghu}
\end{equation}
\end{enumerate}
\end{lemma}
\begin{proof}
We verify that both representations follow from each other. First, suppose that a vector $\bm{x}$ satisfies double inequality \eqref{I-alphagxalphah}. Left multiplication of the right inequality at \eqref{I-alphagxalphah} by $\bm{g}\bm{h}^{-}$ yields $\bm{g}\bm{h}^{-}\bm{x}\leq\alpha\bm{g}\bm{h}^{-}\bm{h}=\alpha\bm{g}$. Considering the left inequality, we see that $\bm{x}\geq\alpha\bm{g}\geq\bm{g}\bm{h}^{-}\bm{x}$, and hence write $\bm{x}=\bm{x}\oplus\bm{g}\bm{h}^{-}\bm{x}$. With $\bm{u}=\bm{x}$, we obtain $\bm{x}=\bm{u}\oplus\bm{g}\bm{h}^{-}\bm{u}=(\bm{I}\oplus\bm{g}\bm{h}^{-})\bm{u}$, which gives \eqref{E-xIghu}.

Now assume that $\bm{x}$ is a vector given by \eqref{E-xIghu}. Take the scalar $\alpha=\bm{h}^{-}\bm{u}$ and write $\bm{x}=(\bm{I}\oplus\bm{g}\bm{h}^{-})\bm{u}\geq\bm{g}\bm{h}^{-}\bm{u}=\alpha\bm{g}$, which provides the left inequality in \eqref{I-alphagxalphah}. Furthermore, it follows from the inequalities $\bm{h}\geq\bm{g}$ and $\bm{h}\bm{h}^{-}\geq\bm{I}$ that $\bm{x}=(\bm{I}\oplus\bm{g}\bm{h}^{-})\bm{u}\leq(\bm{h}\bm{h}^{-}\oplus\bm{g}\bm{h}^{-})\bm{u}=(\bm{h}\oplus\bm{g})\bm{h}^{-}\bm{u}=\bm{h}\bm{h}^{-}\bm{u}=\alpha\bm{h}$, and therefore, the right inequality is valid as well.
\end{proof}

Fig.~\ref{F-ESS12} offers a graphical illustration in terms of $\mathbb{R}_{\max,+}^{2}$ for the representation lemma. An example set defined by inequality \eqref{I-alphagxalphah} is depicted on the left. The rectangle, formed by horizontal and vertical lines drawn through the ends of the vectors $\bm{g}=(g_{1},g_{2})^{T}$ and $\bm{h}=(h_{1},h_{2})^{T}$, shows the boundaries of the set given by \eqref{I-alphagxalphah} with $\alpha=0$. The whole set is then represented as the strip area between thick hatched lines, which is covered when the rectangle shifts at $45^{\circ}$ to the axes in response to the variation of $\alpha$.
\begin{figure}[ht]
\setlength{\unitlength}{1mm}
\begin{picture}(50,45)

\put(0,20){\vector(1,0){50}}
\put(25,0){\vector(0,1){45}}


\put(25,20){\thicklines\vector(-1,1){5}}

\put(20,25){\line(1,0){20}}
\put(20,25){\line(0,1){5}}

\put(1.9,12){\thicklines\line(1,1){29}}
\put(2.1,12){\thicklines\line(1,1){29}}
\multiput(2,12)(1,1){29}{\line(1,0){1}}

\put(25,20){\thicklines\vector(3,2){15}}

\put(40,30){\line(-1,0){20}}
\put(40,30){\line(0,-1){11}}

\put(16.9,2){\thicklines\line(1,1){29}}
\put(17.1,2){\thicklines\line(1,1){29}}
\multiput(17,2)(1,1){29}{\line(-1,0){1}}





\put(16,22){$\bm{g}$}

\put(40,32){$\bm{h}$}

\put(27,32){$h_{2}$}
\put(39,15){$h_{1}$}


\end{picture}
\hspace{20\unitlength}
\begin{picture}(50,45)

\put(0,20){\vector(1,0){50}}
\put(25,0){\vector(0,1){45}}


\put(25,20){\thicklines\vector(-1,1){5}}


\put(1.9,12){\thicklines\line(1,1){29}}
\put(2.1,12){\thicklines\line(1,1){29}}
\multiput(2,12)(1,1){29}{\line(1,0){1}}

\put(25,20){\thicklines\vector(3,2){15}}


\put(16.9,2){\thicklines\line(1,1){29}}
\put(17.1,2){\thicklines\line(1,1){29}}
\multiput(17,2)(1,1){29}{\line(-1,0){1}}

\put(4,9){\line(1,1){30}}

\put(5,10){\line(1,0){20}}
\put(10,15){\line(0,1){5}}

\put(25,20){\thicklines\vector(-3,-1){15}}

\put(25,20){\thicklines\vector(-2,-1){20}}

\put(15,30){$\bm{g}$}
\put(40,32){$\bm{h}$}

\put(1,3){$h_{1}^{-1}\bm{g}$}
\put(1,24){$h_{2}^{-1}\bm{g}$}



\end{picture}
\caption{An example set defined in $\mathbb{R}_{\max,+}^{2}$ by conditions \eqref{I-alphagxalphah} (left) and \eqref{E-xIghu} (right).}
\label{F-ESS12}
\end{figure}
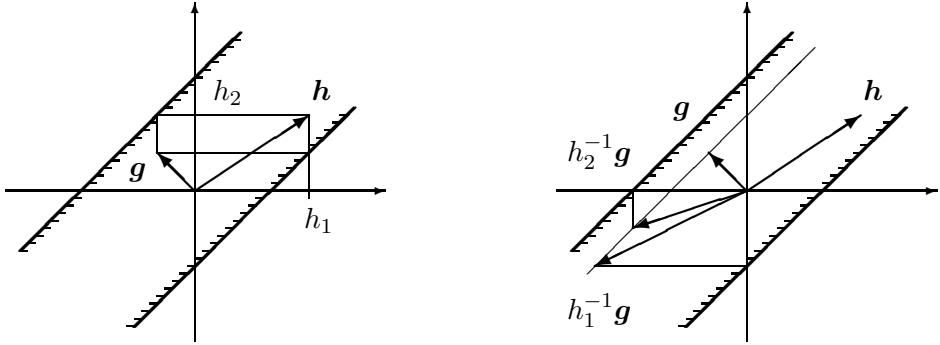

According to representation \eqref{E-xIghu}, the same area is shown on the right as the linear span of the columns in the matrix $\bm{I}\oplus\bm{g}\bm{h}^{-}$, where $\bm{g}\bm{h}^{-}=(h_{1}^{-1}\bm{g},h_{2}^{-1}\bm{g})$.

\section{Tropical optimization problems}
\label{S-TOP}

In this section, we present optimization problems of interest, and outline previous results on their solution. The problems are formulated in general terms of an arbitrary tropical semifield to minimize and maximize a nonlinear function defined by conjugate transposition of vectors. This function arises in various problems, which involve the minimization and maximization of span seminorm and find applications in optimal scheduling. Examples include problems in just-in-time manufacturing \cite{Krivulin2013Explicit} and machine scheduling \cite{Krivulin2016Amaximization}.

First, we consider the minimization problem: given a matrix $\bm{A}\in\mathbb{X}^{m\times n}$ and vectors $\bm{p}\in\mathbb{X}^{m}$, $\bm{q}\in\mathbb{X}^{n}$, find regular vectors $\bm{x}\in\mathbb{X}^{n}$ that
\begin{equation}
\begin{aligned}
&
\text{minimize}
&&
\bm{q}^{-}\bm{x}(\bm{A}\bm{x})^{-}\bm{p}.
\end{aligned}
\label{P-minqxAxp}
\end{equation}

Note that substitution of $\alpha\bm{x}$, where $\alpha\ne\mathbb{0}$, for the vector $\bm{x}$ does not affect the objective function, and thus all solutions of \eqref{P-minqxAxp} are scale-invariant.

A partial solution to the problem formulated in a slightly different form was given in \cite{Krivulin2013Explicit} as follows.
\begin{lemma}
Let $\bm{A}$ be a row-regular matrix, $\bm{p}$ be nonzero and $\bm{q}$ regular vectors. Then, the minimum value in problem \eqref{P-minqxAxp} is equal to $\Delta=(\bm{A}\bm{q})^{-}\bm{p}$, and attained at any vector $\bm{x}=\alpha\bm{q}$, where $\alpha>\mathbb{0}$.
\end{lemma}

Furthermore, we formulate the maximization problem with the same objective function to find regular vectors $\bm{x}\in\mathbb{X}^{n}$ that  
\begin{equation}
\begin{aligned}
&
\text{maximize}
&&
\bm{q}^{-}\bm{x}(\bm{A}\bm{x})^{-}\bm{p}.
\end{aligned}
\label{P-maxqxAxp}
\end{equation}

A complete solution of the problem based on the results in \cite{Krivulin2016Amaximization} can be described in the form of the next statement.
\begin{lemma}
\label{L-maxqxAxp}
Let $\bm{A}=(\bm{a}_{j})$ be a matrix with regular columns $\bm{a}_{j}=(a_{ij})$, $\bm{p}=(p_{j})$ and $\bm{q}=(q_{j})$ be regular vectors. Then, the maximum value in problem \eqref{P-maxqxAxp} is equal to $\Delta
=\bm{q}^{-}\bm{A}^{-}\bm{p}$, and attained if and only if the vector $\bm{x}=(x_{j})$ has the elements
\begin{equation}
x_{k}
=
\alpha\bm{a}_{k}^{-}\bm{p},
\quad
x_{j}
\leq
\alpha a_{sj}^{-1}p_{s},
\quad
j\ne k,
\label{I-xkalphaakp}
\end{equation}
for all $\alpha>\mathbb{0}$ and indices $k$ and $s$ given by
\begin{equation}
k
=
\arg\max_{1\leq j\leq n}q_{j}^{-1}\bm{a}_{j}^{-}\bm{p},
\quad
s
=
\arg\max_{1\leq i\leq m}a_{ik}^{-1}p_{i}.
\label{E-kqjajp-saikpi}
\end{equation}
\end{lemma} 

In the subsequent sections, we use matrix sparsification to derive and represent complete solutions to both problems in a compact vector form.

\section{Solution to the minimization problem}
\label{S-SMinP}

We start with a complete solution of the minimization problem given by \eqref{P-minqxAxp}. As the first step, we follow the arguments in \cite{Krivulin2013Explicit} to find the minimum value, and to derive a partial solution of the problem. Then, we reduce the problem to the solution of simultaneous equation and inequality, and investigate properties of the solution set.

To extend the partial solution, we suggest an entry-wise thresholding procedure to sparsify the matrix in the problem. Then, we apply the sparsified matrix to find new solutions, and illustrate the result with an example.

Furthermore, we describe all solutions as a family of sets, each defined by a matrix obtained from the sparsified matrix of the problem. Next, a backtracking procedure that generates all members in the family of solutions is discussed. Finally, we combine the solutions to provide a direct representation of a complete solution of the problem in a compact closed form.

\subsection{Analysis and characterization of solution}

The next lemma includes the derivation of the partial solution taken from \cite{Krivulin2013Explicit}, which is added to the proof to provide completeness of the argument.
\begin{lemma}
\label{L-minqxAxp}
Let $\bm{A}$ be a row-regular matrix, $\bm{p}$ be nonzero and $\bm{q}$ regular vectors. Then, the minimum value in problem \eqref{P-minqxAxp} is equal to
\begin{equation*}
\Delta
=
(\bm{A}\bm{q})^{-}\bm{p},
\end{equation*}
and all regular vectors $\bm{x}$ that produce this minimum are defined by the system
\begin{equation}
\bm{q}^{-}\bm{x}
=
\alpha,
\quad
\bm{A}\bm{x}
\geq
\alpha\Delta^{-1}\bm{p},
\quad
\alpha
>
\mathbb{0}.
\label{E-qxalpha-I-AxalphaDeltap}
\end{equation}

Specifically, the minimum is attained at any vector $\bm{x}=\alpha\bm{q}$, where $\alpha>\mathbb{0}$.
\end{lemma}

\begin{proof}
To obtain the minimum value of the objective function in problem \eqref{P-minqxAxp}, we derive a lower bound for the function, and then show that this bound is strict.

Suppose that $\bm{x}$ is a regular solution of the problem. Since $\bm{x}\bm{x}^{-}\geq\bm{I}$, we have $(\bm{q}^{-}\bm{x})^{-1}\bm{x}=(\bm{q}^{-}\bm{x}\bm{x}^{-})^{-}\leq\bm{q}$. Next, left multiplication by the matrix $\bm{A}$ gives the inequality $(\bm{q}^{-}\bm{x})^{-1}\bm{A}\bm{x}\leq\bm{A}\bm{q}$, where both sides are regular vectors. Finally, conjugate transposition followed by right multiplication by the vector $\bm{p}$ yields the lower bound $\bm{q}^{-}\bm{x}(\bm{A}\bm{x})^{-}\bm{p}\geq(\bm{A}\bm{q})^{-}\bm{p}=\Delta>\mathbb{0}$.

With $\bm{x}=\bm{q}$, the objective function becomes $\bm{q}^{-}\bm{x}(\bm{A}\bm{x})^{-}\bm{p}=(\bm{A}\bm{q})^{-}\bm{p}=\Delta $, and therefore, $\Delta $ is the minimum value of the problem. 

Considering that all solutions are scale-invariant, we see that not only the vector $\bm{q}$, but also any vector $\bm{x}=\alpha\bm{q}$ with nonzero $\alpha$ solves the problem. 

Furthermore, all vectors $\bm{x}$ that yield the minimum must satisfy the equation
$$
\bm{q}^{-}\bm{x}(\bm{A}\bm{x})^{-}\bm{p}
=
\Delta.
$$

To examine the equation, we put $\alpha=\bm{q}^{-}\bm{x}>\mathbb{0}$, and rewrite it in an equivalent form as the system
$$
\bm{q}^{-}\bm{x}
=
\alpha,
\quad
(\bm{A}\bm{x})^{-}\bm{p}
=
\alpha^{-1}\Delta.
$$

It follows from the first equation that each solution $\bm{x}$ meets the condition $\bm{x}\leq\alpha\bm{q}$. To see this, we consider the inequality $\bm{q}^{-}\bm{x}\leq\alpha$ as a consequence of the equation, and then apply Lemma~\ref{L-Ax-d} to solve the last inequality for $\bm{x}$.

Next, we examine the second equation, which can be replaced by two opposite inequalities $(\bm{A}\bm{x})^{-}\bm{p}\leq\alpha^{-1}\Delta$ and $(\bm{A}\bm{x})^{-}\bm{p}\geq\alpha^{-1}\Delta$. An application of Lemma~\ref{L-Ax-d} to the first inequality with $\bm{p}$ as the unknown vector gives the inequality $\bm{p}\leq\alpha^{-1}\Delta\bm{A}\bm{x}$, which is equivalent to $\bm{A}\bm{x}\geq\alpha\Delta^{-1}\bm{p}$. At the same time, the condition $\bm{x}\leq\alpha\bm{q}$ leads to $(\bm{A}\bm{x})^{-}\bm{p}\geq\alpha^{-1}(\bm{A}\bm{q})^{-}\bm{p}=\alpha^{-1}\Delta$, and thus makes the second inequality superfluous.

As a result, the system under investigation reduces to the form of \eqref{E-qxalpha-I-AxalphaDeltap}.
\end{proof}

The following statement is an important consequence of Lemma~\ref{L-minqxAxp}.

\begin{corollary}
\label{C-minqxAxp}
Let $\bm{A}$ be a row-regular matrix, $\bm{p}$ be nonzero and $\bm{q}$ regular vectors. Then, the set of regular solutions of problem \eqref{P-minqxAxp} is closed under vector addition and scalar multiplication.
\end{corollary}
\begin{proof}
To verify the statement, we only consider addition, since scalar multiplication is examined in a similar manner. Suppose vectors $\bm{x}$ and $\bm{y}$ are regular solutions of problem \eqref{P-minqxAxp}, such that the vector $\bm{x}$ satisfies system \eqref{E-qxalpha-I-AxalphaDeltap}, whereas $\bm{y}$ solves the system
$$
\bm{q}^{-}\bm{y}
=
\beta,
\quad
\bm{A}\bm{y}
\geq
\beta\Delta^{-1}\bm{p},
\quad
\beta
>
\mathbb{0}.
$$

Furthermore, we immediately verify that $\bm{q}^{-}(\bm{x}\oplus\bm{y})=\bm{q}^{-}\bm{x}\oplus\bm{q}^{-}\bm{y}=\alpha\oplus\beta$ and $\bm{A}(\bm{x}\oplus\bm{y})=\bm{A}\bm{x}\oplus\bm{A}\bm{y}\geq(\alpha\oplus\beta)\Delta^{-1}\bm{p}$, which shows that the sum $\bm{x}\oplus\bm{y}$ also obeys system \eqref{E-qxalpha-I-AxalphaDeltap}, where $\alpha$ is replaced by $\alpha\oplus\beta$.
\end{proof}

Note that an application of Lemma~\ref{L-alphagxalphah} provides problem \eqref{P-minqxAxp} with another representation of the solution $\bm{x}=\alpha\bm{q}$ in the form
$$
\bm{x}
=
(\bm{I}\oplus\bm{q}\bm{q}^{-})\bm{u},
\quad
\bm{u}
>
\bm{0}.
$$

However, this representation is not sufficiently different from that offered by Lemma~\ref{L-minqxAxp}. Indeed, considering that the vector $\bm{q}$ is regular, we immediately obtain $\bm{x}=(\bm{I}\oplus\bm{q}\bm{q}^{-})\bm{u}=\bm{q}\bm{q}^{-}\bm{u}=\alpha\bm{q}$, where we take $\alpha=\bm{q}^{-}\bm{u}$.

\subsection{Matrix sparsification}

To derive an extended solution of problem \eqref{P-minqxAxp}, we use a procedure that sets each entry of the matrix $\bm{A}$ to $\mathbb{0}$ if it is below a threshold value determined by both this matrix and the vectors $\bm{p}$ and $\bm{q}$, and leaves the entry unchanged otherwise. The next result introduces the sparsified matrix, and shows that the sparsification does not affect the solution of the problem.

\begin{lemma}
\label{L-Sparse}
Let $\bm{A}=(a_{ij})$ be a row-regular matrix, $\bm{p}=(p_{i})$ be a nonzero, $\bm{q}=(q_{j})$ be a regular vector, and $\Delta=(\bm{A}\bm{q})^{-}\bm{p}$. Define the sparsified matrix $\widehat{\bm{A}}=(\widehat{a}_{ij})$ with the entries
\begin{equation}
\widehat{a}_{ij}
=
\begin{cases}
a_{ij},
&
\text{if $a_{ij}\geq\Delta^{-1}p_{i}$}q_{j}^{-1};
\\
\mathbb{0},
&
\text{otherwise}.
\end{cases}
\label{E-aijprime}
\end{equation}

Then, replacing the matrix $\bm{A}$ by $\widehat{\bm{A}}$ does not change the solutions of problem \eqref{P-minqxAxp}.
\end{lemma}
\begin{proof}
We first verify that the sparsification retains the minimum value given by Lemma~\ref{L-minqxAxp} in the form $\Delta=(\bm{A}\bm{q})^{-}\bm{p}$. We define indices $k$ and $s$ by the conditions
$$
k
=
\arg\max_{1\leq i\leq m}(a_{i1}q_{1}\oplus\cdots\oplus a_{in}q_{n})^{-1}p_{i},
\quad
s
=
\arg\max_{1\leq j\leq n}a_{kj}q_{j},
$$
and then represent $\Delta$ by using the scalar equality
$$
\Delta 
=
\bigoplus_{i=1}^{m}(a_{i1}q_{1}\oplus\cdots\oplus a_{in}q_{n})^{-1}p_{i}
=
(a_{k1}q_{1}\oplus\cdots\oplus a_{kn}q_{n})^{-1}p_{k}
=
(a_{ks}q_{s})^{-1}p_{k}.
$$

The regularity of $\bm{A}$ and $\bm{q}$ guarantees that $a_{i1}q_{1}\oplus\cdots\oplus a_{in}q_{n}>\mathbb{0}$ for all $i$. Since $\bm{p}$ is nonzero, we see that $\Delta>\mathbb{0}$ as well as that $a_{ks}>\mathbb{0}$ and $p_{k}>\mathbb{0}$.

Let us examine an arbitrary row $i$ in the matrix $\bm{A}$. The above equality for $\Delta$ yields the inequality $\Delta\geq(a_{i1}q_{1}\oplus\cdots\oplus a_{in}q_{n})^{-1}p_{i}$, which is equivalent to the inequality $a_{i1}q_{1}\oplus\cdots\oplus a_{in}q_{n}\geq\Delta^{-1}p_{i}$. Because the order defined by the relation $\leq$ is assumed total, the last inequality is valid if and only if the condition $a_{ij}q_{j}\geq\Delta^{-1}p_{i}$ holds for some $j$.

Thus, we conclude that each row $i$ of $\bm{A}$ has at least one entry $a_{ij}$ to satisfy the inequality
\begin{equation}
a_{ij}
\geq
\Delta^{-1}p_{i}q_{j}^{-1}.
\label{I-aijDelat1piqj1}
\end{equation}

Now consider row $k$ in the matrix $\bm{A}$ to verify the inequality $a_{kj}\leq\Delta^{-1}p_{k}q_{j}^{-1}$ for all $j$. Indeed, provided that $a_{kj}=\mathbb{0}$, the inequality is trivially true. If $a_{kj}>\mathbb{0}$, then we have $(a_{kj}q_{j})^{-1}p_{k}\geq(a_{k1}q_{1}\oplus\cdots\oplus a_{kn}q_{n})^{-1}p_{k}=\Delta$, which gives the desired inequality. Since $\Delta=(a_{ks}q_{s})^{-1}p_{k}$, we see that row $k$ has entries, which turns inequality \eqref{I-aijDelat1piqj1} into an equality, but no entry for which \eqref{I-aijDelat1piqj1} becomes strict.

Suppose that inequality \eqref{I-aijDelat1piqj1} fails for some $i$ and $j$. Provided that $p_{i}>\mathbb{0}$, we write $a_{ij}<\Delta^{-1}p_{i}q_{j}^{-1}\leq(a_{i1}q_{1}\oplus\cdots\oplus a_{in}q_{n})q_{j}^{-1}$, which gives the inequality $a_{ij}q_{j}<a_{i1}q_{1}\oplus\cdots\oplus a_{in}q_{n}$. The last inequality means that decreasing $a_{ij}q_{j}$ through lowering of $a_{ij}$ down to $\mathbb{0}$ does not affect the value of $a_{i1}q_{1}\oplus\cdots\oplus a_{in}q_{n}$, and hence the value of $\Delta\geq(a_{i1}q_{1}\oplus\cdots\oplus a_{in}q_{n})^{-1}p_{i}$. Note that if $p_{i}=\mathbb{0}$, then $\Delta$ does not depend at all on the entries in row $i$, including, certainly, $a_{ij}$.

We now verify that all entries $a_{ij}$ that do not satisfy inequality \eqref{I-aijDelat1piqj1} can be set to $\mathbb{0}$ without affecting not only the minimum value $\Delta$, but also the regular solutions of problem \eqref{P-minqxAxp}. First, note that all vectors $\bm{x}=(x_{j})$ providing the minimum in the problem are determined by the equation $\bm{q}^{-}\bm{x}(\bm{A}\bm{x})^{-}\bm{p}=\Delta$.

We represent this equation in the scalar form
$$
(q_{1}^{-1}x_{1}\oplus\cdots\oplus q_{n}^{-1}x_{n})\bigoplus_{i=1}^{m}(a_{i1}x_{1}\oplus\cdots\oplus a_{in}x_{n})^{-1}p_{i}
=
\Delta,
$$
which yields that $a_{i1}x_{1}\oplus\cdots\oplus a_{in}x_{n}\geq\Delta^{-1}(q_{1}^{-1}x_{1}\oplus\cdots\oplus q_{n}^{-1}x_{n})p_{i}$ for all $i$.

Assume the matrix $\bm{A}$ to have an entry, say $a_{ij}$, that satisfies the condition $a_{ij}<\Delta^{-1}p_{i}q_{j}^{-1}$, and thereby violates inequality \eqref{I-aijDelat1piqj1}. Provided that $p_{i}=\mathbb{0}$, the condition leads to the equality $a_{ij}=\mathbb{0}$. Suppose that $p_{i}>\mathbb{0}$, and write
$$
a_{ij}x_{j}
<
\Delta^{-1}p_{i}q_{j}^{-1}x_{j}
\leq
\Delta^{-1}(q_{1}^{-1}x_{1}\oplus\cdots\oplus q_{n}^{-1}x_{n})p_{i}
\leq
a_{i1}x_{1}\oplus\cdots\oplus a_{in}x_{n}.
$$

This inequality implies that, for each solution of the above equation, the term $a_{ij}x_{j}$ does not contribute to the value of the entire sum $a_{i1}x_{1}\oplus\cdots\oplus a_{in}x_{n}$ involved in the calculation of the left-hand side of the equation. Therefore, we can set $a_{ij}$ to $\mathbb{0}$ without altering the solutions of this equation. 

It remains to see that setting the entries $a_{ij}$, which do not satisfy inequality \eqref{I-aijDelat1piqj1}, to $\mathbb{0}$ is equivalent to the replacement of $\bm{A}$ by the matrix $\widehat{\bm{A}}$.
\end{proof}

The matrix obtained after the sparsification procedure for problem \eqref{P-minqxAxp} is referred to below as the sparsified matrix of the problem.

Note that the sparsification of the matrix $\bm{A}$ according to definition \eqref{E-aijprime} is actually determined by the threshold matrix $\Delta^{-1}\bm{p}\bm{q}^{-}$, which contains the threshold values for corresponding entries of $\bm{A}$.

Let $\widehat{\bm{A}}$ be the sparsified matrix for $\bm{A}$, based on the threshold matrix $\Delta^{-1}\bm{p}\bm{q}^{-}$. Then, it follows directly from \eqref{E-aijprime} that the inequality $\widehat{\bm{A}}^{-}\leq\Delta\bm{q}\bm{p}^{-}$ is valid.

\subsection{Extended solution set}

We now assume problem \eqref{P-minqxAxp} already has a sparsified matrix. Under this assumption, we use the characterization of solutions given by Lemma~\ref{L-minqxAxp} to improve the partial solution provided by this lemma by further extending the solution set.  

\begin{theorem}
\label{T-minqxAxp0}
Let $\bm{A}$ be a row-regular sparsified matrix of problem \eqref{P-minqxAxp} with a nonzero vector $\bm{p}$ and a regular vector $\bm{q}$. 

Then, the minimum value in the problem is equal to $\Delta=(\bm{A}\bm{q})^{-}\bm{p}$, and attained at any vector $\bm{x}$ given by the conditions 
\begin{equation}
\alpha\Delta^{-1}\bm{A}^{-}\bm{p}
\leq
\bm{x}
\leq
\alpha\bm{q},
\quad
\alpha
>
\mathbb{0};
\label{I-alphadelta1Apxq}
\end{equation}
or, equivalently, by the conditions
\begin{equation}
\bm{x}
=
(\bm{I}\oplus\Delta^{-1}\bm{A}^{-}\bm{p}\bm{q}^{-})\bm{u},
\quad
\bm{u}
>
\bm{0}.
\label{E-xIdelta1Apqu}
\end{equation}
\end{theorem}

\begin{proof}
It follows from Lemma~\ref{L-minqxAxp} and Lemma~\ref{L-Sparse} that the minimum value, given by $\Delta=(\bm{A}\bm{q})^{-}\bm{p}$, and the regular solutions do not change after sparsification.

Considering that, by Lemma~\ref{L-minqxAxp}, all regular solutions are defined by system \eqref{E-qxalpha-I-AxalphaDeltap}, we need to show that each vector $\bm{x}$, which satisfies \eqref{I-alphadelta1Apxq}, also solves \eqref{E-qxalpha-I-AxalphaDeltap}.

Note that the set of vectors given by inequality \eqref{I-alphadelta1Apxq} is not empty. Indeed, as the matrix $\bm{A}$ is sparsified, the inequality $\bm{A}^{-}\leq\Delta\bm{q}\bm{p}^{-}$ holds. Consequently, we obtain $\Delta^{-1}\bm{A}^{-}\bm{p}\leq\Delta^{-1}\Delta\bm{q}\bm{p}^{-}\bm{p}=\bm{q}$, which results in $\alpha\Delta^{-1}\bm{A}^{-}\bm{p}\leq\alpha\bm{q}$.

By using properties of conjugate transposition, we have $(\bm{A}\bm{q}\bm{q}^{-})^{-}=\bm{q}(\bm{A}\bm{q})^{-}$ and $\bm{q}\bm{q}^{-}\geq\bm{I}$. Then, we write $\bm{q}^{-}\bm{A}^{-}\geq\bm{q}^{-}(\bm{A}\bm{q}\bm{q}^{-})^{-}=\bm{q}^{-}\bm{q}(\bm{A}\bm{q})^{-}=(\bm{A}\bm{q})^{-}$. After left multiplication of \eqref{I-alphadelta1Apxq} by $\bm{q}^{-}$, we obtain
$$
\alpha
=
\alpha\Delta^{-1}(\bm{A}\bm{q})^{-}\bm{p}
\leq
\alpha\Delta^{-1}\bm{q}^{-}\bm{A}^{-}\bm{p}
\leq
\bm{q}^{-}\bm{x}
\leq
\alpha\bm{q}^{-}\bm{q}
=
\alpha,
$$
and thus arrive at the first equality at \eqref{E-qxalpha-I-AxalphaDeltap}.

In addition, it follows from the row regularity of $\bm{A}$ and the left inequality in \eqref{I-alphadelta1Apxq} that $\bm{A}\bm{x}\geq\alpha\Delta^{-1}\bm{A}\bm{A}^{-}\bm{p}\geq\alpha\Delta^{-1}\bm{p}$, which gives the second inequality at \eqref{E-qxalpha-I-AxalphaDeltap}.

Finally, application of Lemma~\ref{L-alphagxalphah} provides the representation of the solution in the form of \eqref{E-xIdelta1Apqu}, which completes the proof.
\end{proof}

\begin{example}
\label{X-Apq}
As an illustration, we examine problem \eqref{P-minqxAxp}, where $m=n=2$, in the framework of the semifield $\mathbb{R}_{\max,+}$ with the matrix and vectors given by
$$
\bm{A}
=
\left(
\begin{array}{cc}
2 & 0
\\
4 & 1
\end{array}
\right),
\quad
\bm{p}
=
\left(
\begin{array}{c}
5
\\
2
\end{array}
\right),
\quad
\bm{q}
=
\left(
\begin{array}{c}
1
\\
2
\end{array}
\right).
$$

We start with the evaluation of the minimum value by calculating
$$
\bm{A}\bm{q}
=
\left(
\begin{array}{c}
3
\\
5
\end{array}
\right),
\quad
\Delta
=
(\bm{A}\bm{q})^{-}\bm{p}
=
2.
$$

Next, we find the threshold and sparsified matrices. With $\mathbb{0}=-\infty$, we write
$$
\Delta^{-1}\bm{p}\bm{q}^{-}
=
\left(
\begin{array}{rr}
2 & 1
\\
-1 & -2
\end{array}
\right),
\quad
\widehat{\bm{A}}
=
\left(
\begin{array}{cc}
2 & \mathbb{0}
\\
4 & 1
\end{array}
\right),
\quad
\Delta^{-1}\widehat{\bm{A}}^{-}\bm{p}\bm{q}^{-}
=
\left(
\begin{array}{rr}
0 & -1
\\
-2 & -3
\end{array}
\right).
$$

The solution given by \eqref{I-alphadelta1Apxq} is represented as follows:
$$
\alpha\bm{x}^{\prime}
\leq
\bm{x}
\leq
\alpha\bm{x}^{\prime\prime},
\quad
\bm{x}^{\prime}
=
\Delta^{-1}\widehat{\bm{A}}^{-}\bm{p}
=
\left(
\begin{array}{r}
1
\\
-1
\end{array}
\right),
\quad
\bm{x}^{\prime\prime}
=
\bm{q}
=
\left(
\begin{array}{c}
1
\\
2
\end{array}
\right),
\quad
\alpha\in\mathbb{R}.
$$

By applying \eqref{E-xIdelta1Apqu}, we obtain the solution in the alternative form
$$
\bm{x}
=
\bm{S}
\bm{u},
\quad
\bm{S}
=
\bm{I}
\oplus
\Delta^{-1}\widehat{\bm{A}}^{-}\bm{p}\bm{q}^{-}
=
\left(
\begin{array}{rr}
0 & -1
\\
-2 & 0
\end{array}
\right),
\quad
\bm{u}\in\mathbb{R}^{2}.
$$

A graphical illustration of the solution is given in Fig.~\ref{F-PECS}, which shows both the known partial solution by Lemma~\ref{L-minqxAxp} (left), and the new extended solution provided by Theorem~\ref{T-minqxAxp0} (middle). In the left picture, the solution is depicted as a thick line drawn through the end point of the vector $\bm{q}$ at $45^{\circ}$ to the axes.

The extended solution in the middle is represented by a strip between two hatched thick lines, which includes the previous solution as the upper boundary. Due to \eqref{I-alphadelta1Apxq}, this strip is drawn as the area covered when the vertical segment between the ends of the vectors $\bm{x}^{\prime}$ and $\bm{x}^{\prime\prime}$ shifts at $45^{\circ}$ to the axes. Solution \eqref{E-xIdelta1Apqu} is depicted as the linear span of columns in the matrix $\bm{S}=(\bm{s}_{1},\bm{s}_{2})$.
\begin{figure}[ht]
\setlength{\unitlength}{1mm}
\begin{picture}(30,45)

\put(0,20){\vector(1,0){30}}
\put(15,0){\vector(0,1){45}}


\put(15,20){\thicklines\vector(1,2){8}}

\put(2.9,16){\thicklines\line(1,1){23}}
\put(3.1,16){\thicklines\line(1,1){23}}

\put(15,20){\thicklines\vector(-3,2){4.75}}

\put(26,35){$\bm{q}$}

\put(7,25){$\bm{x}$}


\end{picture}
\hspace{8\unitlength}
\begin{picture}(37,45)

\put(0,20){\vector(1,0){37}}
\put(15,0){\vector(0,1){45}}


\put(15,20){\thicklines\vector(1,2){8}}

\put(2.9,16){\thicklines\line(1,1){23}}
\put(3.1,16){\thicklines\line(1,1){23}}
\multiput(3,16)(1,1){23}{\line(1,0){1}}

\put(15,20){\thicklines\vector(1,-1){8}}

\put(11.9,1){\thicklines\line(1,1){23}}
\put(12.1,1){\thicklines\line(1,1){23}}
\multiput(12,1)(1,1){23}{\line(-1,0){1}}

\put(15,20){\thicklines\vector(0,-1){16}}

\put(15,20){\thicklines\vector(-1,0){8}}

\put(15,20){\thicklines\vector(3,2){10}}

\put(25,34){$\bm{x}^{\prime\prime}=\bm{q}$}

\put(25,9){$\bm{x}^{\prime}$}

\put(9,5){$\bm{s}_{1}$}
\put(3,22){$\bm{s}_{2}$}

\put(26,27){$\bm{x}$}


\end{picture}
\hspace{8\unitlength}
\begin{picture}(35,45)

\put(0,20){\vector(1,0){35}}
\put(15,20){\vector(0,1){25}}


\put(15,20){\thicklines\vector(1,2){8}}

\put(2.9,16){\thicklines\line(1,1){23}}
\put(3.1,16){\thicklines\line(1,1){23}}
\multiput(3,16)(1,1){23}{\line(1,0){1}}

\put(15,20){\thicklines\vector(2,-1){8}}

\put(15,20){\thicklines\line(0,-1){10}}
\multiput(15,8)(0,-2){3}{\circle*{1}}
\put(15,2){\thicklines\vector(0,-1){2}}

\put(15,20){\thicklines\vector(-1,0){8}}

\put(9,0){$\bm{s}_{1}$}
\put(3,22){$\bm{s}_{2}$}

\put(25,13){$\bm{x}$}


\end{picture}
\caption{Partial (left), extended (middle), and complete (right) solutions.}
\label{F-PECS}
\end{figure}
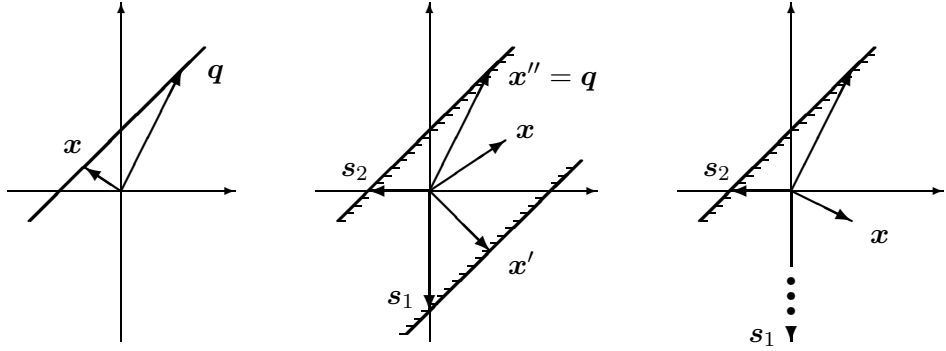
\end{example}

\subsection{Complete solution}

We are now in a position to derive a complete solution to the problem. The next result offers a simple way to describe all solutions to problem \eqref{P-minqxAxp}. 

\begin{theorem}
\label{T-minqxAxp1}
Let $\bm{A}$ be a row-regular sparsified matrix for problem \eqref{P-minqxAxp} with a nonzero vector $\bm{p}$ and a regular vector $\bm{q}$. Let $\mathcal{A}$ be the set of matrices obtained from $\bm{A}$ by fixing one nonzero entry in each row and setting the other ones to $\mathbb{0}$.

Then, the minimum value in \eqref{P-minqxAxp} is equal to $\Delta=(\bm{A}\bm{q})^{-}\bm{p}$, and all regular solutions $\bm{x}$ are given by the conditions 
\begin{equation}
\alpha\Delta^{-1}\bm{A}_{1}^{-}\bm{p}
\leq
\bm{x}
\leq
\alpha\bm{q},
\quad
\alpha
>
\mathbb{0},
\quad
\bm{A}_{1}\in\mathcal{A};
\label{I-alphaDelta1A1pxalphaq}
\end{equation}
or, equivalently, by the conditions
\begin{equation}
\bm{x}
=
(\bm{I}\oplus\Delta^{-1}\bm{A}_{1}^{-}\bm{p}\bm{q}^{-})\bm{u},
\quad
\bm{u}
>
\bm{0},
\quad
\bm{A}_{1}\in\mathcal{A}.
\label{E-xIA1pqu}
\end{equation}
\end{theorem}
\begin{proof}
It follows from Lemma~\ref{L-minqxAxp} that all solutions of problem \eqref{P-minqxAxp} are defined by system \eqref{E-qxalpha-I-AxalphaDeltap}. Therefore, to prove the theorem, we need to show that each solution of system \eqref{E-qxalpha-I-AxalphaDeltap} is a solution of \eqref{I-alphaDelta1A1pxalphaq} with some matrix $\bm{A}_{1}\in\mathcal{A}$, and vice versa.

Consider any matrix $\bm{A}_{1}\in\mathcal{A}$, and note that it is row-regular. Moreover, the inequalities $\bm{A}_{1}\leq\bm{A}$ and $\bm{A}_{1}^{-}\leq\bm{A}^{-}$ hold. In the same way as in Theorem~\ref{T-minqxAxp0}, we see that, since $\bm{A}_{1}^{-}\leq\bm{A}^{-}\leq\Delta\bm{q}\bm{p}^{-}$, the double inequality at \eqref{I-alphaDelta1A1pxalphaq} has solutions.

Let $\bm{x}$ be a solution to system \eqref{E-qxalpha-I-AxalphaDeltap}. First, we take the inequality $\bm{A}\bm{x}\geq\alpha\Delta^{-1}\bm{p}$, and examine every corresponding scalar inequality to determine the maximal summand on the left-hand side. Clearly, there is a matrix $\bm{A}_{1}\in\mathcal{A}$ with nonzero entries that are located in each row to match these maximal summands. With this matrix, the inequality can be replaced by $\bm{A}_{1}\bm{x}\geq\alpha\Delta^{-1}\bm{p}$ without loss of solution. At the same time, the matrix $\bm{A}_{1}$ has exactly one nonzero entry in each row, and thus obeys the inequality $\bm{A}_{1}^{-}\bm{A}_{1}\leq\bm{I}$. After right multiplication by $\bm{x}$, we obtain $\bm{x}\geq\bm{A}_{1}^{-}\bm{A}_{1}\bm{x}\geq\alpha\Delta^{-1}\bm{A}_{1}^{-}\bm{p}$, which gives the left inequality in \eqref{I-alphaDelta1A1pxalphaq}. The right inequality in \eqref{I-alphaDelta1A1pxalphaq} directly follows from the equality $\bm{q}^{-}\bm{x}=\alpha$ at \eqref{E-qxalpha-I-AxalphaDeltap}.

Next, we suppose that a vector $\bm{x}$ satisfies \eqref{I-alphaDelta1A1pxalphaq} with some matrix $\bm{A}_{1}\in\mathcal{A}$, and verify that $\bm{x}$ also solves system \eqref{E-qxalpha-I-AxalphaDeltap}. By using the same argument as in Theorem~\ref{T-minqxAxp0}, we have $\bm{q}^{-}\bm{A}_{1}^{-}\geq(\bm{A}_{1}\bm{q})^{-}\geq(\bm{A}\bm{q})^{-}$, and then obtain the equality at \eqref{E-qxalpha-I-AxalphaDeltap}. Considering that $\bm{A}\bm{A}_{1}^{-}\geq\bm{I}$, we take the left inequality at \eqref{I-alphaDelta1A1pxalphaq} to write $\bm{A}\bm{x}\geq\alpha\Delta^{-1}\bm{A}\bm{A}_{1}^{-}\bm{p}\geq\alpha\Delta^{-1}\bm{p}$, which yields the inequality at \eqref{E-qxalpha-I-AxalphaDeltap}. An application of Lemma~\ref{L-alphagxalphah} completes the proof.
\end{proof}

Note that the solution sets defined by different matrices from the set $\mathcal{A}$ in Theorem~\ref{T-minqxAxp1} can have nonempty intersection, as shown in the next example.

\begin{example}
\label{X-Apq-S}
Suppose that the matrix in Example~\ref{X-Apq} is replaced by its sparsified matrix, and consider the problem with
$$
\bm{A}
=
\left(
\begin{array}{cc}
2 & \mathbb{0}
\\
4 & 1
\end{array}
\right),
\quad
\bm{p}
=
\left(
\begin{array}{c}
5
\\
2
\end{array}
\right),
\quad
\bm{q}
=
\left(
\begin{array}{c}
1
\\
2
\end{array}
\right).
$$

Since the sparsification of the matrix does not change the minimum in the problem, we still have $\Delta=(\bm{A}\bm{q})^{-}\bm{p}=2$.

Consider the set $\mathcal{A}$, which is formed of the matrices obtained from $\bm{A}$ by keeping only one nonzero entry in each row. This set consists of two matrices
$$
\bm{A}_{1}
=
\left(
\begin{array}{cc}
2 & \mathbb{0}
\\
4 & \mathbb{0}
\end{array}
\right),
\quad
\bm{A}_{2}
=
\left(
\begin{array}{cc}
2 & \mathbb{0}
\\
\mathbb{0} & 1
\end{array}
\right).
$$

Let us write the solutions defined by these matrices in the form of \eqref{E-xIA1pqu}. First, we calculate the matrices
$$
\Delta^{-1}\bm{A}_{1}^{-}\bm{p}\bm{q}^{-}
=
\left(
\begin{array}{rr}
0 & -1
\\
\mathbb{0} & \mathbb{0}
\end{array}
\right),
\quad
\Delta^{-1}\bm{A}_{2}^{-}\bm{p}\bm{q}^{-}
=
\left(
\begin{array}{rr}
0 & -1
\\
-2 & -3
\end{array}
\right).
$$

Using the first matrix yields the solution
$$
\bm{x}
=
\bm{S}_{1}
\bm{u},
\quad
\bm{S}_{1}
=
\bm{I}
\oplus
\Delta^{-1}\bm{A}_{1}^{-}\bm{p}\bm{q}^{-}
=
\left(
\begin{array}{rr}
0 & -1
\\
\mathbb{0} & 0
\end{array}
\right),
\quad
\bm{u}\in\mathbb{R}^{2}.
$$

The second solution coincides with that obtained in Example~\ref{X-Apq} in the form
$$
\bm{x}
=
\bm{S}_{2}
\bm{u},
\quad
\bm{S}_{2}
=
\bm{I}
\oplus
\Delta^{-1}\bm{A}_{2}^{-}\bm{p}\bm{q}^{-}
=
\left(
\begin{array}{rr}
0 & -1
\\
-2 & 0
\end{array}
\right),
\quad
\bm{u}\in\mathbb{R}^{2}.
$$

The first solution is displayed in Fig.~\ref{F-PECS} (right) as the half-plane below the thick hatched line. Clearly, this area completely covers the strip region in Fig.~\ref{F-PECS} (middle), offered by the extended solution.
\end{example}

\subsection{Backtracking procedure for generating solutions}

Consider a backtracking search procedure that finds all solutions to problem \eqref{P-minqxAxp} with the sparsified matrix $\bm{A}$ in an economical way. To generate all matrices in $\mathcal{A}$, the procedure examines each row in the matrix $\bm{A}$ to fix one nonzero entry in the row, and to set the other entries to zeros. After selecting a nonzero entry in the current row, the subsequent rows are modified to reduce the number of remaining alternatives. Then, a nonzero entry in the next row of the modified matrix is fixed if any exists, and the procedure continues repeatedly.

Suppose that every row of the modified matrix has exactly one nonzero entry. This matrix is considered as a solution matrix $\bm{A}_{1}\in\mathcal{A}$, and stored in a solution list. Furthermore, if the modified matrix has zero rows, it does not provide a solution. In either case, the procedure returns to roll back all last modifications, and to fix the next nonzero entry in the current row if there is any, or goes back to the previous row otherwise. The procedure is completed when no more nonzero entries in the first row of the matrix $\bm{A}$ left to select. 

To describe the technique used to reduce search, suppose that the procedure, which has fixed one nonzero entry in each of the rows $1,\ldots,i-1$, currently selects a nonzero entry in row $i$ of the modified matrix $\widetilde{\bm{A}}$, say the entry $\widetilde{a}_{ij}$ in column $j$, whereas the other entries in the row are set to zero. 

Any solution vector $\bm{x}$ must satisfy the inequality $\widetilde{\bm{A}}\bm{x}\geq\alpha\Delta^{-1}\bm{p}$ in system \eqref{E-qxalpha-I-AxalphaDeltap}. Specifically, the scalar inequality for row $i$, where only the entry $\widetilde{a}_{ij}$ is nonzero, reads $\widetilde{a}_{ij}x_{j}\geq\alpha\Delta^{-1}p_{i}$, or, equivalently, $x_{j}\geq\alpha\Delta^{-1}\widetilde{a}_{ij}^{-1}p_{i}$. If $p_{i}>\mathbb{0}$, then the inequality determines a lower bound for $x_{j}$ in the solution under construction.

Assuming $p_{i}>\mathbb{0}$, consider the entries of column $j$ in rows $k=i+1,\ldots,n$. Provided that the condition $\widetilde{a}_{kj}\geq\widetilde{a}_{ij}p_{i}^{-1}p_{k}$ is satisfied for row $k$, we write $\widetilde{a}_{kj}x_{j}\geq\alpha\Delta^{-1}\widetilde{a}_{ij}p_{i}^{-1}p_{k}\widetilde{a}_{ij}^{-1}p_{i}\geq\alpha\Delta^{-1}p_{k}$, which means that the inequality at \eqref{E-qxalpha-I-AxalphaDeltap} for this row is valid regardless of $x_{l}$ for $l\ne j$. In this case, further examination of nonzero entries $\widetilde{a}_{kl}$ in row $k$ cannot impose new constraints on the element $x_{l}$ in the vector $\bm{x}$, and thus is not needed. These entries can be set to zeros without affecting the inequality, which may decrease the number of search alternatives.

\begin{example}
\label{X-Apq-B}
As a simple illustration of the technique, we return to Example~\ref{X-Apq-S}, where the initial sparsified matrix and its further sparsifications are given by
$$
\bm{A}
=
\left(
\begin{array}{cc}
2 & \mathbb{0}
\\
4 & 1
\end{array}
\right),
\quad
\bm{A}_{1}
=
\left(
\begin{array}{cc}
2 & \mathbb{0}
\\
4 & \mathbb{0}
\end{array}
\right),
\quad
\bm{A}_{2}
=
\left(
\begin{array}{cc}
2 & \mathbb{0}
\\
\mathbb{0} & 1
\end{array}
\right).
$$  

The procedure first fixes the entry $a_{11}=2$. Since $a_{21}=4$ is greater than $a_{11}p_{1}^{-1}p_{2}=-1$, the procedure sets $a_{22}$ to $\mathbb{0}$, which immediately excludes the matrix $\bm{A}_{2}$ from further consideration, and hence reduces the analysis to $\bm{A}_{1}$.
\end{example}

\subsection{Representation of complete solution in closed form}

A complete solution to problem~\eqref{P-minqxAxp} can be expressed in a closed form as follows.

\begin{theorem}
\label{T-minqxAxp}
Let $\bm{A}$ be a row-regular sparsified matrix for problem \eqref{P-minqxAxp} with a nonzero vector $\bm{p}$ and a regular vector $\bm{q}$. Denote by $\mathcal{A}$ the set of matrices obtained from $\bm{A}$ by fixing one nonzero entry in each row, and setting the other ones to $\mathbb{0}$.

Let $\bm{S}$ be the matrix, which is formed by putting together all columns of the matrices $\bm{S}_{1}=\bm{I}\oplus\Delta^{-1}\bm{A}_{1}^{-}\bm{p}\bm{q}^{-}$ for every $\bm{A}_{1}\in\mathcal{A}$, and $\bm{S}_{0}$ be a matrix whose columns comprise a maximal linear independent system of the columns in $\bm{S}$. 

Then, the minimum value in \eqref{P-minqxAxp} is equal to $\Delta=(\bm{A}\bm{q})^{-}\bm{p}$, and all regular solutions are given by
$$
\bm{x}
=
\bm{S}_{0}\bm{v},
\quad
\bm{v}
>
\bm{0}.
$$
\end{theorem}
\begin{proof}
Suppose that the set $\mathcal{A}$ consists of $k$ elements, which can be enumerated as $\bm{A}_{1},\ldots,\bm{A}_{k}$. For each $\bm{A}_{i}\in\mathcal{A}$, we define the matrix $\bm{S}_{i}=
\bm{I}\oplus\Delta^{-1}\bm{A}_{i}^{-}\bm{p}\bm{q}^{-}$.

First, note that, by Theorem~\ref{T-minqxAxp1}, the set of vectors $\bm{x}$ that solve problem \eqref{P-minqxAxp} is the union of subsets, each of which corresponds to one index $i=1,\ldots,k$, and contains the vectors given by $\bm{x}=\bm{S}_{i}\bm{u}_{i}$, where $\bm{u}_{i}>\bm{0}$ is a vector.

We now verify that all solutions to the problem can also be represented as
\begin{equation}
\bm{x}
=
\bm{S}_{1}\bm{u}_{1}\oplus\cdots\oplus\bm{S}_{k}\bm{u}_{k},
\quad
\bm{u}_{1},\ldots,\bm{u}_{k}>\bm{0}. 
\label{E-xB1u1Bkuk}
\end{equation}

Indeed, any solution provided by Theorem~\ref{T-minqxAxp1} can be written in the form of \eqref{E-xB1u1Bkuk}. At the same time, since the solution set is closed under vector addition and scalar multiplication by Corollary~\ref{C-minqxAxp}, any vector $\bm{x}$ given by \eqref{E-xB1u1Bkuk} solves the problem. Therefore, representation \eqref{E-xB1u1Bkuk} describes all solutions of the problem.

With the matrix $\bm{S}=(\bm{S}_{1},\ldots,\bm{S}_{k})$ and the vector $\bm{u}=(\bm{u}_{1}^{T},\ldots,\bm{u}_{k}^{T})^{T}$, we rewrite \eqref{E-xB1u1Bkuk} in the form
$$
\bm{x}
=
\bm{S}\bm{u},
\quad
\bm{u}
>
\bm{0},
$$
which specifies each solution to be a linear combination of columns in $\bm{S}$.

Clearly, elimination of a column that linearly depends on some others leaves the linear span of the columns unchanged. By eliminating all dependent columns, we reduce the matrix $\bm{S}$ to the matrix $\bm{S}_{0}$ to express any solution to the problem by a linear combination of columns in $\bm{S}_{0}$ as $\bm{x}=\bm{S}_{0}\bm{v}$, where $\bm{v}>\bm{0}$ is a vector of appropriate size.
\end{proof}

\begin{example}
We again consider results of Example~\ref{X-Apq-S} to examine the matrices
$$
\bm{S}_{1}
=
\left(
\begin{array}{rr}
0 & -1
\\
\mathbb{0} & 0
\end{array}
\right),
\quad
\bm{S}_{2}
=
\left(
\begin{array}{rr}
0 & -1
\\
-2 & 0
\end{array}
\right).
$$

We take the dissimilar columns from $\bm{S}_{1}$ and $\bm{S}_{2}$, and denote them by
$$
\bm{s}_{1}
=
\left(
\begin{array}{r}
0
\\
\mathbb{0}
\end{array}
\right),
\quad
\bm{s}_{2}
=
\left(
\begin{array}{r}
-1
\\
0
\end{array}
\right),
\quad
\bm{s}_{3}
=
\left(
\begin{array}{r}
0
\\
-2
\end{array}
\right).
$$

Next, we put these columns together to form the matrix
$$
\bm{S}
=
\left(
\begin{array}{rrr}
\bm{s}_{1} & \bm{s}_{2} & \bm{s}_{3}
\end{array}
\right)
=
\left(
\begin{array}{rrr}
0 & -1 & 0
\\
\mathbb{0} & 0 & -2
\end{array}
\right).
$$

Furthermore, we examine the matrix $\bm{S}_{1}=(\bm{s}_{1},\bm{s}_{2})$ to calculate $\delta(\bm{S}_{1},\bm{s}_{3})$, and then to apply Lemma~\ref{E-AbAb1}. Since we have
$$
(\bm{s}_{3}^{-}\bm{S}_{1})^{-}
=
\bm{S}_{1}(\bm{s}_{3}^{-}\bm{S}_{1})^{-}
=
\left(
\begin{array}{r}
0
\\
-2
\end{array}
\right),
\quad
\delta(\bm{S}_{1},\bm{s}_{3})
=
(\bm{S}_{1}(\bm{s}_{3}^{-}\bm{S}_{1})^{-})^{-}\bm{s}_{3}
=
0
=
\mathbb{1},
$$ 
the column $\bm{s}_{3}$ is linearly dependent on the others, and thus can be removed.

Considering that the columns $\bm{s}_{1}$ and $\bm{s}_{2}$ are obviously not collinear, none of them can be further eliminated. With $\bm{S}_{0}=\bm{S}_{1}$, a complete solution to the problem is given by $\bm{x}=\bm{S}_{0}\bm{v}$, where $\bm{v}>\bm{0}$, and depicted in Fig.~\ref{F-PECS} (right).
\end{example}

\section{Solution to the maximization problem}
\label{S-SMaxP}

We now consider the maximization problem given by \eqref{P-maxqxAxp} and its solution offered by Lemma~\ref{L-maxqxAxp}. Note that the lemma represents the solution vectors by the conditions at \eqref{I-xkalphaakp} in scalar terms rather than in a vector form. Below, we show how the application of Lemma~\ref{L-alphagxalphah} and the use of sparsified matrices enable the transformation of the scalar solution into a compact vector form, similar to that of the above solution to the minimization problem.

The next result offers a vector representation of solution given by Lemma~\ref{L-maxqxAxp}.
\begin{theorem}
\label{T-maxqxAxp}
Let $\bm{A}=(\bm{a}_{j})$ be a matrix with regular columns $\bm{a}_{j}=(a_{ij})$, $\bm{p}=(p_{j})$ and $\bm{q}=(q_{j})$ be regular vectors. Let $\bm{A}_{sk}$ denote the matrix obtained from $\bm{A}$ by fixing the entry $a_{sk}$ for some indices $s$ and $k$, and replacing the other entries by $\mathbb{0}$.

Then, the maximum value in problem \eqref{P-maxqxAxp} is equal to $\Delta=\bm{q}^{-}\bm{A}^{-}\bm{p}$, and all regular solutions are given by
\begin{equation*}
\bm{x}
=
(\bm{I}\oplus\bm{A}_{sk}^{-}\bm{A})\bm{u},
\quad
\bm{u}>\bm{0},
\end{equation*}
for all indices $k$ and $s$ defined by the conditions
\begin{equation*}
k
=
\arg\max_{1\leq j\leq m}q_{j}^{-1}\bm{a}_{j}^{-}\bm{p},
\quad
s
=
\arg\max_{1\leq i\leq n}a_{ik}^{-1}p_{i}.
\end{equation*}
\end{theorem} 
\begin{proof}
Let us consider conditions \eqref{I-xkalphaakp} and \eqref{E-kqjajp-saikpi}, and note that $\bm{a}_{k}^{-}\bm{p}=a_{sk}^{-1}p_{s}$. Furthermore, we define the vector $\bm{g}=(g_{j})$ with the elements $g_{k}=a_{sk}^{-1}p_{s}$ and $g_{k}=\mathbb{0}$ for all $j\ne k$, and the vector $\bm{h}=(h_{j})$ with $h_{j}=a_{sj}^{-1}p_{s}$ for all $j$.

Now, condition \eqref{I-xkalphaakp} takes the form of the double inequality
$$
\alpha\bm{g}
\leq
\bm{x}
\leq
\alpha\bm{h},
\quad
\alpha
>
\mathbb{0},
$$
which, by Lemma~\ref{L-alphagxalphah}, can be equivalently represented as
$$
\bm{x}
=
(\bm{I}\oplus\bm{g}\bm{h}^{-})\bm{u},
\quad
\bm{u}
>
\bm{0}.
$$

It remains to see that $\bm{g}\bm{h}^{-}=\bm{A}_{sk}^{-}\bm{A}$, which completes the proof.
\end{proof}

\begin{example}
Let us apply the theorem to solve the maximization problem with the objective function defined as in Example~\ref{X-Apq}.

First, we have to evaluate the maximum $\Delta$ and determine the indices $k$ and $s$. We calculate
$$
\bm{a}_{1}^{-}\bm{p}
=
3,
\quad
\bm{a}_{2}^{-}\bm{p}
=
5,
\quad
q_{1}^{-1}\bm{a}_{1}^{-}\bm{p}
=
2,
\quad
q_{2}^{-1}\bm{a}_{2}^{-}\bm{p}
=
3,
$$
from which it follows that $\Delta=3$ and $k=2$. Next, we obtain
$$
a_{12}^{-1}p_{1}
=
5,
\quad
a_{22}^{-1}p_{2}
=
1,
$$
and thus conclude that $s=1$.

Furthermore, we calculate the matrices
$$
\bm{A}_{12}
=
\left(
\begin{array}{cc}
\mathbb{0} & 0
\\
\mathbb{0} & \mathbb{0}
\end{array}
\right),
\quad
\bm{A}_{12}^{-}\bm{A}
=
\left(
\begin{array}{cc}
\mathbb{0} & \mathbb{0}
\\
2 & 0
\end{array}
\right),
\quad
\bm{S}
=
\bm{I}\oplus\bm{A}_{12}^{-}\bm{A}
=
\left(
\begin{array}{cc}
0 & \mathbb{0}
\\
2 & 0
\end{array}
\right),
$$
and finally obtain the solution in the form $\bm{x}=\bm{S}\bm{u}$, where $\bm{u}>\bm{0}$.

Fig.~\ref{F-MMS} demonstrates this solution (left) together with the solution of the minimization problem obtained before (right). The obtained solution is generated by the columns of the matrix $\bm{S}=(\bm{s}_{1},\bm{s}_{2})$, and takes the form of the upper half-plane above the hatched thick line.
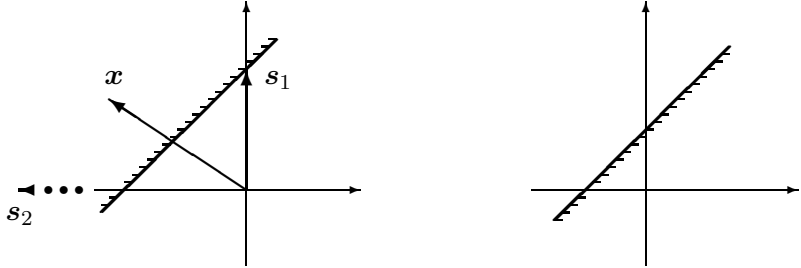
\begin{figure}[ht]
\setlength{\unitlength}{1mm}
\begin{center}
\begin{picture}(50,40)

\put(35,10){\vector(1,0){15}}
\put(35,10){\vector(0,1){25}}


\put(35,10){\thicklines\vector(0,1){16}}

\put(15.9,7){\thicklines\line(1,1){23}}
\put(16.1,7){\thicklines\line(1,1){23}}
\multiput(16,8)(1,1){23}{\line(1,0){1}}

\put(35,10){\line(-1,0){20}}
\multiput(13,10)(-2,0){3}{\circle*{1}}
\put(7,10){\thicklines\vector(-1,0){2}}

\put(35,10){\line(0,-1){10}}

\put(35,10){\thicklines\vector(-3,2){18}}

\put(37,24){$\bm{s}_{1}$}
\put(3,6){$\bm{s}_{2}$}

\put(16,24){$\bm{x}$}


\end{picture}
\hspace{20\unitlength}
\begin{picture}(35,40)

\put(0,10){\vector(1,0){35}}
\put(15,10){\vector(0,1){25}}



\put(2.9,6){\thicklines\line(1,1){23}}
\put(3.1,6){\thicklines\line(1,1){23}}
\multiput(3,6)(1,1){23}{\line(1,0){1}}


\put(15,10){\line(0,-1){10}}






\end{picture}

\end{center}
\caption{Solutions of maximization (left) and minimization (right) problems.}
\label{F-MMS}
\end{figure}
\end{example}

\section{Application to scheduling problems}
\label{S-ASP}

The aim of this section is to apply the results obtained to solve real-world problems from project (machine) scheduling, which serve motivational and illustrative purposes for the study. We start with the description of a general time-constrained scheduling model in the usual notation, and formulate example problems to find optimal schedules under given temporal constraints. The problems arise in just-in-time and limited-resource scheduling, and find applications in the analysis of real-world processes with time synchronization in manufacturing, transportation networks, and parallel data processing (see, e.g. \cite{Demeulemeester2002Project,Neumann2003Project,Tkindt2006Multicriteria,Vanhoucke2012Project} for further details and applications).

We represent the scheduling problems in terms of tropical algebra, and then use the previous results to obtain complete solutions to the problems. A simple but representative numerical example illustrates the results.

\subsection{Model description and problem formulation}

Consider a project that consists of $n$ activities (jobs, tasks) operating in parallel under the temporal constraints in the form of start-start, start-finish and finish-start precedence relations, and finish deadlines. Two problems are of interest: first, to develop a schedule, in which all activities finish as much simultaneously as possible, and second, to find a schedule, which spreads the finish times of activities over time as much as possible.

To describe the temporal constraints for each activity $i=1,\ldots,n$ and the scheduling objective, we denote by $x_{i}$ and $y_{i}$, respectively, the start and finish times to be scheduled. Let $a_{ij}$ be the minimum possible time lag between the start of activity $j$ and the finish of $i$. If this time lag is not defined in the project, we assume $a_{ij}=-\infty$. The start-finish constraints of activity $i$ are given by the inequalities $a_{ij}+x_{j}\leq y_{i}$ for all $j=1,\ldots,n$.

We assume each activity to finish as soon as all its start-finish constraints are satisfied, and therefore, at least one of these inequalities must hold as an equality. As a result, we can combine the inequalities into one equality
\begin{equation*}
\max_{1\leq j\leq n}(a_{ij}+x_{j})
=
y_{i}.
\end{equation*}

Furthermore, we denote by $b_{ij}$ the minimum time lag between the start of activity $j$ and the start of $i$, and put $b_{ij}=-\infty$ if this lag is not specified. The start-start constraints are given for all $j$ by the inequalities $b_{ij}+x_{j}\leq x_{i}$, which are equivalent to one inequality
\begin{equation*}
\max_{1\leq j\leq n}(b_{ij}+x_{j})
\leq
x_{i}.
\end{equation*}

Denote by $c_{ij}$ the minimum time lag between the finish of activity $j$ and the start of $i$, and set it to $-\infty$ if undefined. The finish-start constraints take the form of the inequalities $c_{ij}+y_{j}\leq x_{i}$ for all $j$, or of one inequality
\begin{equation*}
\max_{1\leq j\leq n}(c_{ij}+y_{j})
\leq
x_{i}.
\end{equation*}

Let $f_{i}$ be the deadline for activity $i$, which specifies the latest possible time to finish. The deadline provides an upper boundary for the finish time, which is given by
\begin{equation*}
y_{i}
\leq
f_{i}.
\end{equation*}

We now define the scheduling objectives to minimize and maximize the maximum deviation between the finish times of the activities. The maximum deviation of the finish times is given by
$$
\max_{1\leq i\leq n}y_{i}
+
\min_{1\leq i\leq n}y_{i}
=
\max_{1\leq i\leq n}y_{i}
+
\max_{1\leq i\leq n}(-y_{i}).
$$

The problem of minimizing the maximum deviation of finish times, subject to start-finish, start-start, finish-start and deadline constraints is as follows: given $a_{ij}$, $b_{ij}$, $c_{ij}$ and $f_{i}$, determine the unknowns $x_{i}$ and $y_{i}$, that 
\begin{equation}
\begin{aligned}
&
\text{minimize}
&&
\max_{1\leq i\leq n}y_{i}
+
\max_{1\leq i\leq n}(-y_{i}),
\\
&
\text{subject to}
&&
\max_{1\leq j\leq n}(a_{ij}+x_{j})
=
y_{i},
\quad
\max_{1\leq j\leq n}(b_{ij}+x_{j})
\leq
x_{i},
\\
&&&
\max_{1\leq j\leq n}(c_{ij}+y_{j})
\leq
x_{i},
\quad
y_{i}
\leq
f_{i},
\quad
i=1,\ldots,n.
\end{aligned}
\label{P-minyiyi-aijxjyi-bijxjxj-gixi-cijyjxi-yifi}
\end{equation}

The maximization problem is formulated in much the same way:
\begin{equation}
\begin{aligned}
&
\text{maximize}
&&
\max_{1\leq i\leq n}y_{i}
+
\max_{1\leq i\leq n}(-y_{i}),
\\
&
\text{subject to}
&&
\max_{1\leq j\leq n}(a_{ij}+x_{j})
=
y_{i},
\quad
\max_{1\leq j\leq n}(b_{ij}+x_{j})
\leq
x_{i},
\\
&&&
\max_{1\leq j\leq n}(c_{ij}+y_{j})
\leq
x_{i},
\quad
y_{i}
\leq
f_{i},
\quad
i=1,\ldots,n.
\end{aligned}
\label{P-maxyiyi-aijxjyi-bijxjxj-gixi-cijyjxi-yifi}
\end{equation}

\subsection{Algebraic solution to the minimization problem}

Since the formulation of problem \eqref{P-minyiyi-aijxjyi-bijxjxj-gixi-cijyjxi-yifi} involves only operations of maximum, addition and additive inversion, the problem can be well represented in terms of the idempotent semifield $\mathbb{R}_{\max,+}$. We replace the arithmetic operations by those of the semifield $\mathbb{R}_{\max,+}$ to rewrite problem \eqref{P-minyiyi-aijxjyi-bijxjxj-gixi-cijyjxi-yifi} as 
\begin{equation*}
\begin{aligned}
&
\text{minimize}
&&
\bigoplus_{i=1}^{n}y_{i}
\bigoplus_{j=1}^{n}y_{j}^{-1},
\\
&
\text{subject to}
&&
\bigoplus_{j=1}^{n}a_{ij}x_{j}
=
y_{i},
\quad
\bigoplus_{j=1}^{n}b_{ij}x_{j}
\leq
x_{i},
\\
&&&
\bigoplus_{j=1}^{n}c_{ij}y_{j}
\leq
x_{i},
\quad
y_{i}
\leq
f_{i},
\quad
i=1,\ldots,n.
\end{aligned}
\end{equation*}

Furthermore, we introduce the square matrices
\begin{equation*}
\bm{A}
=
(a_{ij}),
\quad
\bm{B}
=
(b_{ij}),
\quad
\bm{C}
=
(c_{ij}),
\end{equation*}
and the vectors
\begin{equation*}
\bm{x}
=
(x_{i}),
\quad
\bm{y}
=
(y_{i}),
\quad
\bm{f}
=
(f_{i}).
\end{equation*}

With the matrix-vector notation, the problem of interest becomes
\begin{equation}
\begin{aligned}
&
\text{minimize}
&&
\bm{1}^{T}\bm{y}\bm{y}^{-}\bm{1},
\\
&
\text{subject to}
&&
\bm{A}\bm{x}
=
\bm{y},
\quad
\bm{B}\bm{x}
\leq
\bm{x},
\\
&&&
\bm{C}\bm{y}
\leq
\bm{x},
\quad
\bm{y}
\leq
\bm{f}.
\end{aligned}
\label{P-min1yy1-Axy-Bxx-Cyx-yf}
\end{equation}

The following result provides a solution to the problem.

\begin{theorem}
\label{T-min1yy1-Axy-Bxx-Cyx-yf}
Let $\bm{A}$ be a regular matrix, $\bm{B}$ and $\bm{C}$ be matrices such that $\mathop\mathrm{Tr}(\bm{B}\oplus\bm{C}\bm{A})\leq\mathbb{1}$, and $\bm{f}$ be a regular vector. Define the matrix $\bm{D}=\bm{A}(\bm{B}\oplus\bm{C}\bm{A})^{\ast}$ with the columns $\bm{d}_{j}=(d_{ij})$, and denote $\Delta=(\bm{D}(\bm{1}^{T}\bm{D})^{-})^{-}\bm{1}$.

Define the sparsified matrix $\widehat{\bm{D}}=(\widehat{d}_{ij})$ with the entries given by the condition 
\begin{equation*}
\widehat{d}_{ij}
=
\begin{cases}
d_{ij},
&
\text{if $d_{ij}\geq\Delta^{-1}\bm{1}^{T}\bm{d}_{j}$};
\\
\mathbb{0},
&
\text{otherwise}.
\end{cases}
\end{equation*}

Let $\bm{S}$ be the matrix, which is formed by putting together the columns of the matrices $\bm{I}\oplus\Delta^{-1}\widehat{\bm{D}}_{1}^{-}\bm{1}\bm{1}^{T}\bm{D}$ for each matrix $\widehat{\bm{D}}_{1}$ that can be obtained from $\widehat{\bm{D}}$ by fixing one nonzero entry in each row and setting the others to zeros. Let $\bm{S}_{0}$ be the matrix obtained from $\bm{S}$ by removing the columns that are linearly dependent on others.

Then, the minimum value in problem \eqref{P-min1yy1-Axy-Bxx-Cyx-yf} is equal to $\Delta=(\bm{D}(\bm{1}^{T}\bm{D})^{-})^{-}\bm{1}$, and all solutions are given by
\begin{equation}
\bm{x}
=
(\bm{B}\oplus\bm{C}\bm{A})^{\ast}\bm{S}_{0}\bm{v},
\quad
\bm{y}
=
\bm{D}\bm{S}_{0}\bm{v},
\quad
\bm{v}
\leq
(\bm{f}^{-}\bm{D}\bm{S}_{0})^{-}.
\label{E-xBCAS0v-yDS0v-vfDS0}
\end{equation}
\end{theorem}
\begin{proof}
We start with the substitution $\bm{y}=\bm{A}\bm{x}$ to eliminate the vector $\bm{y}$ and reduce the problem to finding regular vectors $\bm{x}$ that
\begin{equation*}
\begin{aligned}
&
\text{minimize}
&&
\bm{1}^{T}\bm{A}\bm{x}(\bm{A}\bm{x})^{-}\bm{1},
\\
&
\text{subject to}
&&
\bm{B}\bm{x}
\leq
\bm{x},
\quad
\bm{C}\bm{A}\bm{x}
\leq
\bm{x},
\\
&&&
\bm{A}\bm{x}
\leq
\bm{f}.
\end{aligned}
\end{equation*}

We now combine the first two constraints $\bm{B}\bm{x}\leq\bm{x}$ and $\bm{C}\bm{A}\bm{x}\leq\bm{x}$ into one inequality $(\bm{B}\oplus\bm{C}\bm{A})\bm{x}\leq\bm{x}$, where the matrix satisfies the condition $\mathop\mathrm{Tr}(\bm{B}\oplus\bm{C}\bm{A})\leq\mathbb{1}$. An application of Theorem~\ref{T-Ax-x} to solve this inequality yields the solution $\bm{x}=(\bm{B}\oplus\bm{C}\bm{A})^{\ast}\bm{u}$, where $\bm{u}$ is any regular vector.

We substitute the solution obtained into the objective function and the last inequality constraint, and then use the notation $\bm{D}=\bm{A}(\bm{B}\oplus\bm{C}\bm{A})^{\ast}$. The constraint becomes $\bm{D}\bm{u}\leq\bm{f}$, and has the solution given by Lemma~\ref{L-Ax-d} in the form $\bm{u}\leq(\bm{f}^{-}\bm{D})^{-}$. As a result, we have the problem
\begin{equation}
\begin{aligned}
&
\text{minimize}
&&
\bm{1}^{T}\bm{D}\bm{u}(\bm{D}\bm{u})^{-}\bm{1},
\\
&
\text{subject to}
&&
\bm{u}
\leq
(\bm{f}^{-}\bm{D})^{-}.
\end{aligned}
\label{P-min1DuDu1-ufD}
\end{equation}

To solve problem \eqref{P-min1DuDu1-ufD}, we first ignore the inequality constraint to obtain the solution of the corresponding unconstrained problem, and then reduce the solution to satisfy the constraint. We note that the unconstrained problem has the form of \eqref{P-minqxAxp} with $\bm{q}^{-}$ replaced by $\bm{1}^{T}\bm{D}$, $\bm{A}$ by $\bm{D}$, and $\bm{p}$ by $\bm{1}$, and thus can be solved by Theorem~\ref{T-minqxAxp}.

The solution provided by the theorem defines the minimum in the problem to be $\Delta=(\bm{D}(\bm{1}^{T}\bm{D})^{-})^{-}\bm{1}$, and involves the evaluation of the sparsified matrix $\widehat{\bm{D}}=(\widehat{d}_{ij})$, where $\widehat{d}_{ij}=d_{ij}$ if $d_{ij}\geq\Delta^{-1}\bm{1}^{T}\bm{d}_{j}$, and $\widehat{d}_{ij}=\mathbb{0}$ otherwise.

For each row-regular matrix $\widehat{\bm{D}}_{1}$ that can be obtained from $\widehat{\bm{D}}$ by further replacement of nonzero entries by zeros, we calculate the matrix $\bm{S}_{1}=\bm{I}\oplus\Delta^{-1}\widehat{\bm{D}}_{1}^{-}\bm{1}\bm{1}^{T}\bm{D}$, and then put together the columns of the matrices $\bm{S}_{1}$ to form the matrix $\bm{S}$.

Furthermore, we construct the matrix $\bm{S}_{0}$ by removing those columns from $\bm{S}$, which are linearly dependent on others. Then, by Theorem~\ref{T-minqxAxp}, the solution of the unconstrained problem is given by $\bm{u}=\bm{S}_{0}\bm{v}$, where $\bm{v}$ is any regular vector of appropriate size.

The solution obtained satisfies the condition in problem \eqref{P-min1DuDu1-ufD} if the inequality $\bm{u}=\bm{S}_{0}\bm{v}\leq(\bm{f}^{-}\bm{D})^{-}$ holds. An application of Lemma~\ref{L-Ax-d} yields the solution to this inequality in the form $\bm{v}\leq(\bm{f}^{-}\bm{D}\bm{S}_{0})^{-}$.

Returning to the vectors $\bm{x}$ and $\bm{y}$, we obtain the solution
$$
\bm{x}
=
(\bm{B}\oplus\bm{C}\bm{A})^{\ast}\bm{S}_{0}\bm{v},
\quad
\bm{y}
=
\bm{D}\bm{S}_{0}\bm{v},
\quad
\bm{v}
\leq
(\bm{f}^{-}\bm{D}\bm{S}_{0})^{-},
$$
and thus complete the proof.
\end{proof}

\subsection{Illustrative example}

To illustrate the above result and to demonstrate the computation involved, we present the solution of an example problem of low dimension.

\begin{example}
\label{X-min1yy1-Axy-Bxx-Cyx-yf}
Consider a project that involves $n=3$ activities operating under start-finish, start-start, finish-start and finish deadline temporal constraints given by
\begin{gather*}
\bm{A}
=
\left(
\begin{array}{rrc}
3 & -1 & \mathbb{0}
\\
-2 & 2 & 0
\\
-1 & \mathbb{0} & 4
\end{array}
\right),
\quad
\bm{B}
=
\left(
\begin{array}{crr}
\mathbb{0} & \mathbb{0} & -3
\\
2 & \mathbb{0} & 0
\\
1 & -2 & \mathbb{0}
\end{array}
\right),
\\
\bm{C}
=
\left(
\begin{array}{rcr}
\mathbb{0} & \mathbb{0} & \mathbb{0}
\\
0 & \mathbb{0} & -3
\\
-1 & \mathbb{0} & \mathbb{0}
\end{array}
\right),
\quad
\bm{f}
=
\left(
\begin{array}{c}
7
\\
7
\\
7
\end{array}
\right),
\end{gather*}
where we use the symbol $\mathbb{0}=-\infty$ to simplify the writing.

Let $\bm{x}=(x_{1},x_{2},x_{3})^{T}$ and $\bm{y}=(y_{1},y_{2},y_{3})^{T}$ denote the unknown vectors of start and finish times of activities in the project. The problem is to find vectors $\bm{x}$ and $\bm{y}$ that minimize the maximum deviation of finish times, subject to the given temporal constraints.

To apply Theorem~\ref{T-min1yy1-Axy-Bxx-Cyx-yf} to the problem, we have to verify that the conditions of the theorem hold. First, we note that the matrix $\bm{A}$ and the vector $\bm{f}$ are regular. Next, we obtain the matrices
\begin{equation*}
\bm{C}\bm{A}
=
\left(
\begin{array}{crc}
\mathbb{0} & \mathbb{0} & \mathbb{0}
\\
3 & -1 & 1
\\
2 & -2 & \mathbb{0}
\end{array}
\right),
\quad
\bm{B}\oplus\bm{C}\bm{A}
=
\left(
\begin{array}{crr}
\mathbb{0} & \mathbb{0} & -3
\\
3 & -1 & 1
\\
2 & -2 & \mathbb{0}
\end{array}
\right),
\end{equation*}
and then calculate the powers
\begin{equation*}
(\bm{B}\oplus\bm{C}\bm{A})^{2}
=
\left(
\begin{array}{rrr}
-1 & -5 & \mathbb{0}
\\
3 & -1 & 0
\\
1 & -3 & -1
\end{array}
\right),
\quad
(\bm{B}\oplus\bm{C}\bm{A})^{3}
=
\left(
\begin{array}{rrr}
-2 & -6 & -4
\\
2 & -2 & 0
\\
1 & -3 & -2
\end{array}
\right).
\end{equation*}

After evaluating the traces of the powers, we have $\mathop\mathrm{Tr}(\bm{B}\oplus\bm{C}\bm{A})=-1<0=\mathbb{1}$, and thus conclude that all conditions of Theorem~\ref{T-min1yy1-Axy-Bxx-Cyx-yf} are fulfilled.

Furthermore, we obtain the matrices
\begin{gather*}
(\bm{B}\oplus\bm{C}\bm{A})^{\ast}
=
\bm{I}
\oplus
\bm{B}\oplus\bm{C}\bm{A}
\oplus
(\bm{B}\oplus\bm{C}\bm{A})^{2}
=
\left(
\begin{array}{crr}
0 & -5 & -3
\\
3 & 0 & 1
\\
2 & -2 & 0
\end{array}
\right),
\\
\bm{D}
=
\bm{A}(\bm{B}\oplus\bm{C}\bm{A})^{\ast}
=
\left(
\begin{array}{crc}
3 & -1 & 0
\\
5 & 2 & 3
\\
6 & 2 & 4
\end{array}
\right).
\end{gather*}

To find the minimum in the problem, we successively calculate
\begin{equation*}
\bm{1}^{T}\bm{D}
=
\left(
\begin{array}{ccc}
6 & 2 & 4
\end{array}
\right),
\quad
\bm{D}(\bm{1}^{T}\bm{D})^{-}
=
\left(
\begin{array}{r}
-3
\\
0
\\
0
\end{array}
\right),
\quad
\Delta
=
\bm{D}(\bm{1}^{T}\bm{D})^{-}\bm{1}
=
3.
\end{equation*}

The threshold and sparsified matrices for the matrix $\bm{D}$ take the form 
\begin{equation*}
\Delta^{-1}\bm{1}\bm{1}^{T}\bm{D}
=
\left(
\begin{array}{rrc}
3 & -1 & 1
\\
3 & -1 & 1
\\
3 & -1 & 1
\end{array}
\right),
\quad
\widehat{\bm{D}}
=
\left(
\begin{array}{crc}
3 & -1 & \mathbb{0}
\\
5 & 2 & 3
\\
6 & 2 & 4
\end{array}
\right).
\end{equation*}

We now need to construct the matrices, which can be obtained from the matrix $\widehat{\bm{D}}$ by replacing all but one of the non-zero entries in each row by zeros. To reduce the number of matrices to be examined, we follow the backtracking technique described above. First, we fix the entry $\widehat{d}_{11}=3$ and set $\widehat{d}_{12}$ to $\mathbb{0}$. Since the entries in the second and third rows of the first column satisfy the conditions $\widehat{d}_{21}\geq\widehat{d}_{11}$ and $\widehat{d}_{31}\geq\widehat{d}_{11}$, the other entries in these rows can be set to zeros, which gives the matrix with nonzero entries only in the first column. Using the same argument, we obtain another matrix, where only the second column is nonzero. As a result, we have two matrices
$$
\widehat{\bm{D}}_{1}
=
\left(
\begin{array}{ccc}
3 & \mathbb{0} & \mathbb{0}
\\
5 & \mathbb{0} & \mathbb{0}
\\
6 & \mathbb{0} & \mathbb{0}
\end{array}
\right),
\quad
\widehat{\bm{D}}_{2}
=
\left(
\begin{array}{crc}
\mathbb{0} & -1 & \mathbb{0}
\\
\mathbb{0} & 2 & \mathbb{0}
\\
\mathbb{0} & 2 & \mathbb{0}
\end{array}
\right).
$$
 
Furthermore, we take the conjugate transposes
$$
\widehat{\bm{D}}_{1}^{-}
=
\left(
\begin{array}{rrr}
-3 & -5 & -6
\\
\mathbb{0} & \mathbb{0} & \mathbb{0}
\\
\mathbb{0} & \mathbb{0} & \mathbb{0}
\end{array}
\right),
\quad
\widehat{\bm{D}}_{2}^{-}
=
\left(
\begin{array}{crr}
\mathbb{0} & \mathbb{0} & \mathbb{0}
\\
1 & -2 & -2
\\
\mathbb{0} & \mathbb{0} & \mathbb{0}
\end{array}
\right)
$$
to calculate the matrices
\begin{gather*}
\bm{S}_{1}
=
\bm{I}
\oplus
\Delta^{-1}\widehat{\bm{D}}_{1}^{-}\bm{1}\bm{1}^{T}\bm{D}
=
\left(
\begin{array}{crr}
0 & -4 & -2
\\
\mathbb{0} & 0 & \mathbb{0}
\\
\mathbb{0} & \mathbb{0} & 0
\end{array}
\right),
\\
\bm{S}_{2}
=
\bm{I}
\oplus
\Delta^{-1}\widehat{\bm{D}}_{2}^{-}\bm{1}\bm{1}^{T}\bm{D}
=
\left(
\begin{array}{crr}
0 & \mathbb{0} & \mathbb{0}
\\
4 & 0 & 2
\\
\mathbb{0} & \mathbb{0} & 0
\end{array}
\right).
\end{gather*}

Consider the columns of the matrices $\bm{S}_{1}$ and $\bm{S}_{2}$, and denote them as
\begin{gather*}
\bm{s}_{1}
=
\left(
\begin{array}{c}
0
\\
\mathbb{0}
\\
\mathbb{0}
\end{array}
\right),
\quad
\bm{s}_{2}
=
\left(
\begin{array}{c}
\mathbb{0}
\\
0
\\
\mathbb{0}
\end{array}
\right),
\quad
\bm{s}_{3}
=
\left(
\begin{array}{r}
-2
\\
\mathbb{0}
\\
0
\end{array}
\right),
\quad
\bm{s}_{4}
=
\left(
\begin{array}{c}
\mathbb{0}
\\
2
\\
0
\end{array}
\right),
\\
\bm{s}_{5}
=
\left(
\begin{array}{r}
-4
\\
0
\\
\mathbb{0}
\end{array}
\right),
\quad
\bm{s}_{6}
=
\left(
\begin{array}{c}
0
\\
4
\\
\mathbb{0}
\end{array}
\right).
\end{gather*}

Next, we put the columns together to compose the matrix
\begin{equation*}
\bm{S}
=
\left(
\begin{array}{ccrcrc}
0 & \mathbb{0} & -2 & \mathbb{0} & -4 & 0
\\
\mathbb{0} & 0 & \mathbb{0} & 2 & 0 & 4
\\
\mathbb{0} & \mathbb{0} & 0 & 0 & \mathbb{0} & \mathbb{0}
\end{array}
\right).
\end{equation*}

To eliminate the columns, which are linearly dependent on others, we first note that the columns $\bm{s}_{5}$ and $\bm{s}_{6}$ are collinear, and thus remove the last column. Moreover, it is not difficult to verify using Lemma~\ref{E-AbAb1} that the column $\bm{s}_{5}$ is itself dependent on the first four columns. Indeed, we take the columns $\bm{s}_{1}$, $\bm{s}_{2}$, $\bm{s}_{3}$ and $\bm{s}_{4}$ to form the matrix
\begin{equation*}
\bm{S}_{0}
=
\left(
\begin{array}{ccrc}
0 & \mathbb{0} & -2 & \mathbb{0}
\\
\mathbb{0} & 0 & \mathbb{0} & 2
\\
\mathbb{0} & \mathbb{0} & 0 & 0
\end{array}
\right),
\end{equation*}
and then calculate the vectors
\begin{equation*}
(\bm{s}_{5}^{-}\bm{S}_{0})^{-}
=
\left(
\begin{array}{r}
-4
\\
0
\\
-2
\\
-2
\end{array}
\right),
\quad
\bm{S}_{0}(\bm{s}_{5}^{-}\bm{S}_{0})^{-}
=
\left(
\begin{array}{r}
-4
\\
0
\\
-2
\end{array}
\right).
\end{equation*}

Finally, we have $\delta(\bm{S}_{0},\bm{s}_{5})=(\bm{S}_{0}(\bm{s}_{5}^{-}\bm{S}_{0})^{-})^{-}\bm{s}_{5}=0=\mathbb{1}$, which means that $\bm{s}_{5}$ is linear dependent on columns in $\bm{S}_{0}$.

By applying the same verification technique to the columns of the matrix $\bm{S}_{0}$, we conclude that this matrix has no dependent columns.

We are now in a position to represent the solution given by \eqref{E-xBCAS0v-yDS0v-vfDS0}. We start with the evaluation of the matrices
\begin{equation*}
(\bm{B}\oplus\bm{C}\bm{A})^{\ast}\bm{S}_{0}
=
\left(
\begin{array}{crrr}
0 & -5 & -2 & -3
\\
3 & 0 & 1 & 2
\\
2 & -2 & 0 & 0
\end{array}
\right),
\quad
\bm{A}(\bm{B}\oplus\bm{C}\bm{A})^{\ast}\bm{S}_{0}
=
\left(
\begin{array}{crcc}
3 & -1 & 1 & 1
\\
5 & 2 & 3 & 4
\\
6 & 2 & 4 & 4
\end{array}
\right),
\end{equation*}
and then calculate the vector
\begin{equation*}
(\bm{f}^{-}\bm{A}(\bm{B}\oplus\bm{C}\bm{A})^{\ast}\bm{S}_{0})^{-}
=
\left(
\begin{array}{c}
1
\\
5
\\
3
\\
3
\end{array}
\right).
\end{equation*}

Finally, we introduce the vector $\bm{v}=(v_{1},v_{2},v_{3},v_{4})^{T}$, and write the solution in the form
\begin{equation*}
\bm{x}
=
\left(
\begin{array}{crrr}
0 & -5 & -2 & -3
\\
3 & 0 & 1 & 2
\\
2 & -2 & 0 & 0
\end{array}
\right)
\bm{v},
\quad
\bm{y}
=
\left(
\begin{array}{crcc}
3 & -1 & 1 & 1
\\
5 & 2 & 3 & 4
\\
6 & 2 & 4 & 4
\end{array}
\right)
\bm{v},
\quad
\bm{v}
\leq
\left(
\begin{array}{c}
1
\\
5
\\
3
\\
3
\end{array}
\right).
\end{equation*}

Note that the solution can be simplified considering that the third column of the matrix $(\bm{B}\oplus\bm{C}\bm{A})^{\ast}\bm{S}_{0}$ is collinear with the first, and the forth column is with the second. As a result, we can factorize this matrix as
\begin{equation*}
(\bm{B}\oplus\bm{C}\bm{A})^{\ast}\bm{S}_{0}
=
\left(
\begin{array}{cr}
0 & -5
\\
3 & 0
\\
2 & -2
\end{array}
\right)
\left(
\begin{array}{ccrr}
0 & \mathbb{0} & -2 & -3
\\
\mathbb{0} & 0 & 1 & 2
\end{array}
\right),
\end{equation*}
and then replace the vector $\bm{v}$ by new vector $\bm{w}=(w_{1},w_{2})^{T}$ defined as
\begin{equation*}
\bm{w}
=
\left(
\begin{array}{ccrr}
0 & \mathbb{0} & -2 & -3
\\
\mathbb{0} & 0 & 1 & 2
\end{array}
\right)
\bm{v}.
\end{equation*}

Furthermore, we replace
\begin{equation*}
\bm{A}(\bm{B}\oplus\bm{C}\bm{A})^{\ast}\bm{S}_{0}\bm{v}
=
\left(
\begin{array}{rrc}
3 & -1 & \mathbb{0}
\\
-2 & 2 & 0
\\
-1 & \mathbb{0} & 4
\end{array}
\right)
\left(
\begin{array}{cr}
0 & -5
\\
3 & 0
\\
2 & -2
\end{array}
\right)
\bm{w}
=
\left(
\begin{array}{cr}
3 & -1
\\
5 & 2
\\
6 & 2
\end{array}
\right)
\bm{w},
\end{equation*}
and find
\begin{equation*}
\bm{w}
=
\left(
\begin{array}{ccrr}
0 & \mathbb{0} & -2 & -3
\\
\mathbb{0} & 0 & 1 & 2
\end{array}
\right)
\bm{v}
\leq
\left(
\begin{array}{ccrr}
0 & \mathbb{0} & -2 & -3
\\
\mathbb{0} & 0 & 1 & 2
\end{array}
\right)
\left(
\begin{array}{c}
1
\\
5
\\
3
\\
3
\end{array}
\right)
=
\left(
\begin{array}{c}
1
\\
5
\end{array}
\right).
\end{equation*}

After the change of variables, the solution reduces to
\begin{equation*}
\bm{x}
=
\left(
\begin{array}{cr}
0 & -5
\\
3 & 0
\\
2 & -2
\end{array}
\right)
\bm{w},
\quad
\bm{y}
=
\left(
\begin{array}{cr}
3 & -1
\\
5 & 2
\\
6 & 2
\end{array}
\right)
\bm{w},
\quad
\bm{w}
\leq
\left(
\begin{array}{c}
1
\\
5
\end{array}
\right).
\end{equation*}

In the standard notation, the solution takes the form
\begin{equation*}
\begin{aligned}
x_{1}
&=
\max(w_{1},w_{2}-5),
\\
x_{2}
&=
\max(w_{1}+3,w_{2}),
\\
x_{3}
&=
\max(w_{1}+2,w_{2}-2),
\end{aligned}
\quad
\begin{aligned}
y_{1}
&=
\max(w_{1}+3,w_{2}-1),
\\
y_{2}
&=
\max(w_{1}+5,w_{2}+2),
\\
y_{3}
&=
\max(w_{1}+6,w_{2}+2),
\end{aligned}
\end{equation*}
where
\begin{equation*}
w_{1}
\leq
1,
\quad
w_{2}
\leq
5.
\end{equation*}

Specifically, the latest start and finish times, which correspond to $w_{1}=1$ and $w_{2}=5$, are given by
\begin{equation*}
x_{1}
=
1,
\quad
x_{2}
=
5,
\quad
x_{3}
=
3,
\quad
y_{1}
=
4,
\quad
y_{2}
=
7,
\quad
y_{2}
=
7.
\end{equation*}
\end{example}

\subsection{Solution to the maximization problem}

Similarly as above, we represent problem \eqref{P-maxyiyi-aijxjyi-bijxjxj-gixi-cijyjxi-yifi} in terms of the semifield $\mathbb{R}_{\max,+}$ in the form
\begin{equation}
\begin{aligned}
&
\text{maximize}
&&
\bm{1}^{T}\bm{y}\bm{y}^{-}\bm{1},
\\
&
\text{subject to}
&&
\bm{A}\bm{x}
=
\bm{y},
\quad
\bm{B}\bm{x}
\leq
\bm{x},
\\
&&&
\bm{C}\bm{y}
\leq
\bm{x},
\quad
\bm{y}
\leq
\bm{f}.
\end{aligned}
\label{P-max1yy1-Axy-Bxx-Cyx-yf}
\end{equation}

A complete solution of the problem can be described as follows.
\begin{theorem}
\label{T-max1yy1-Axy-Bxx-Cyx-yf}
Let $\bm{A}$ be a regular matrix, $\bm{B}$ and $\bm{C}$ be matrices such that $\mathop\mathrm{Tr}(\bm{B}\oplus\bm{C}\bm{A})\leq\mathbb{1}$ and the matrix $\bm{D}=\bm{A}(\bm{B}\oplus\bm{C}\bm{A})^{\ast}$ has regular columns $\bm{d}_{j}=(d_{ij})$.

Denote by $\bm{D}_{sk}$ the matrix obtained from $\bm{D}$ by fixing the entry $d_{sk}$ for some indices $s$ and $k$, and by replacing the other entries by $\mathbb{0}$, and let $\bm{R}_{sk}=\bm{I}\oplus\bm{D}_{sk}^{-}\bm{D}$.

Then, the maximum value in problem \eqref{P-max1yy1-Axy-Bxx-Cyx-yf} is equal to $\Delta=\bm{1}^{T}\bm{D}\bm{D}^{-}\bm{1}$, and all solutions are given by
\begin{equation*}
\bm{x}
=
(\bm{B}\oplus\bm{C}\bm{A})^{\ast}\bm{R}_{sk}\bm{v},
\quad
\bm{y}
=
\bm{D}\bm{R}_{sk}\bm{v},
\quad
\bm{v}
\leq
(\bm{f}^{-}\bm{D}\bm{R}_{sk})^{-},
\end{equation*}
for all indices $k$ and $s$ defined by the conditions
\begin{equation*}
k
=
\arg\max_{1\leq j\leq m}\bm{1}^{T}\bm{d}_{j}\bm{d}_{j}^{-}\bm{1},
\quad
s
=
\arg\max_{1\leq i\leq n}d_{ik}^{-1}.
\end{equation*}
\end{theorem}
\begin{proof}
In the same way as in Theorem~\ref{T-min1yy1-Axy-Bxx-Cyx-yf}, we denote $\bm{D}=\bm{A}(\bm{B}\oplus\bm{C}\bm{A})^{\ast}$, and represent the unknown vectors as $\bm{x}=(\bm{B}\oplus\bm{C}\bm{A})^{\ast}\bm{u}$ and $\bm{y}=\bm{D}\bm{u}$, where the vector $\bm{u}$ is the solution of the problem 
\begin{equation*}
\begin{aligned}
&
\text{maximize}
&&
\bm{1}^{T}\bm{D}\bm{u}(\bm{D}\bm{u})^{-}\bm{1},
\\
&
\text{subject to}
&&
\bm{u}
\leq
(\bm{f}^{-}\bm{D})^{-}.
\end{aligned}
\end{equation*}

Furthermore, we apply Theorem~\ref{T-maxqxAxp}, where $\bm{q}^{-}$ is replaced by $\bm{1}^{T}\bm{D}$, $\bm{A}$ by $\bm{D}$ and $\bm{p}$ by $\bm{1}$, to solve the last problem without constraints. We obtain
$$
\bm{u}
=
(\bm{I}\oplus\bm{D}_{sk}^{-}\bm{D})\bm{v}
=
\bm{R}_{sk}\bm{v},
\quad
\bm{v}
>
\bm{0},
$$
where the indices $k$ and $s$ are given by the conditions
$$
k
=
\arg\max_{1\leq j\leq m}\bm{1}^{T}\bm{d}_{j}\bm{d}_{j}^{-}\bm{1},
\quad
s
=
\arg\max_{1\leq i\leq n}d_{ik}^{-1}.
$$

Substitution into the inequality constraint and application of Lemma~\ref{L-Ax-d} yields the inequality
$$
\bm{v}
\leq
(\bm{f}^{-}\bm{D}\bm{R}_{sk})^{-}.
$$

Finally, we represent the solution of the original problem in the form
$$
\bm{x}
=
(\bm{B}\oplus\bm{C}\bm{A})^{\ast}\bm{R}_{sk}\bm{v},
\quad
\bm{y}
=
\bm{D}\bm{R}_{sk}\bm{v},
$$
which gives the desired result.
\end{proof}

\begin{example}
\label{X-max1yy1-Axy-Bxx-Cyx-yf}

Let us consider the project described in Example~\ref{X-min1yy1-Axy-Bxx-Cyx-yf}, and apply Theorem~\ref{T-max1yy1-Axy-Bxx-Cyx-yf} to find a schedule according to the maximization objective.

We take the previously obtained matrices
$$
(\bm{B}\oplus\bm{C}\bm{A})^{\ast}
=
\left(
\begin{array}{crr}
0 & -5 & -3
\\
3 & 0 & 1
\\
2 & -2 & 0
\end{array}
\right),
\quad
\bm{D}
=
\left(
\begin{array}{crc}
3 & -1 & 0
\\
5 & 2 & 3
\\
6 & 2 & 4
\end{array}
\right),
$$
and note that the matrix $\bm{D}$ has only regular columns. Then, we calculate
$$
\bm{1}^{T}\bm{d}_{1}\bm{d}_{1}^{-}\bm{1}
=
3,
\quad
\bm{1}^{T}\bm{d}_{2}\bm{d}_{2}^{-}\bm{1}
=
3,
\quad
\bm{1}^{T}\bm{d}_{3}\bm{d}_{3}^{-}\bm{1}
=
4,
$$
which yields $k=3$. In addition, we have $\Delta=\bm{1}^{T}\bm{D}\bm{D}^{-}\bm{1}=4$.

Considering that $d_{13}^{-1}=0$, $d_{23}^{-1}=-3$ and $d_{33}^{-1}=-4$, we fix $s=1$.

The application of Theorem~~\ref{T-max1yy1-Axy-Bxx-Cyx-yf} requires the calculation of the matrices
$$
\bm{D}_{sk}
=
\left(
\begin{array}{ccc}
\mathbb{0} & \mathbb{0} & 0
\\
\mathbb{0} & \mathbb{0} & \mathbb{0}
\\
\mathbb{0} & \mathbb{0} & \mathbb{0}
\end{array}
\right),
\quad
\bm{R}_{sk}
=
\bm{I}\oplus\bm{D}_{sk}^{-}\bm{D}
=
\left(
\begin{array}{crc}
0 & \mathbb{0} & \mathbb{0}
\\
\mathbb{0} & 0 & \mathbb{0}
\\
3 & -1 & 0
\end{array}
\right).
\quad
$$

Furthermore, we compute and factorize the matrices
\begin{gather*}
(\bm{B}\oplus\bm{C}\bm{A})^{\ast}\bm{R}_{sk}
=
\left(
\begin{array}{crr}
0 & -4 & -3
\\
4 & 0 & 1
\\
3 & -1 & 0
\end{array}
\right)
=
\left(
\begin{array}{c}
0
\\
4
\\
3
\end{array}
\right)
\left(
\begin{array}{crr}
0 & -4 & -3
\end{array}
\right),
\\
\bm{D}\bm{R}_{sk}
=
\left(
\begin{array}{crc}
3 & -1 & 0
\\
6 & 2 & 3
\\
7 & 3 & 4
\end{array}
\right)
=
\left(
\begin{array}{c}
3
\\
6
\\
7
\end{array}
\right)
\left(
\begin{array}{crc}
0 & -4 & -3
\end{array}
\right),
\end{gather*}
and then find the vector
$$
(\bm{f}^{-}\bm{D}\bm{R}_{sk})^{-}
=
\left(
\begin{array}{c}
0
\\
4
\\
3
\end{array}
\right).
$$

Finally, with the vector $\bm{v}=(v_{1},v_{2},v_{3})^{T}$, the solution is written as
$$
\bm{x}
=
\left(
\begin{array}{crr}
0 & -4 & -3
\\
4 & 0 & 1
\\
3 & -1 & 0
\end{array}
\right)
\bm{v},
\quad
\bm{y}
=
\left(
\begin{array}{crc}
3 & -1 & 0
\\
6 & 2 & 3
\\
7 & 3 & 4
\end{array}
\right)
\bm{v},
\quad
\bm{v}
\leq
\left(
\begin{array}{c}
0
\\
4
\\
3
\end{array}
\right).
$$

Assume a scalar $w$ to satisfy the equality
$$
w
=
\left(
\begin{array}{crr}
0 & -4 & -3
\end{array}
\right)
\bm{v},
$$
and note that
$$
w
=
\left(
\begin{array}{crr}
0 & -4 & -3
\end{array}
\right)
\bm{v}
\leq
\left(
\begin{array}{crr}
0 & -4 & -3
\end{array}
\right)
\left(
\begin{array}{c}
0
\\
4
\\
3
\end{array}
\right)
=
0.
$$

We now turn from the vector $\bm{v}$ to the scalar $w$ to represent the solution in a more compact form as
$$
\bm{x}
=
\left(
\begin{array}{c}
0
\\
4
\\
3
\end{array}
\right)
w,
\quad
\bm{y}
=
\left(
\begin{array}{c}
3
\\
6
\\
7
\end{array}
\right)
w,
\quad
w
\leq
0.
$$

In terms of the conventional algebra, the solution becomes
$$
x_{1}
=
w,
\quad
x_{2}
=
w+4,
\quad
x_{3}
=
w+3,
\quad
y_{1}
=
w+3,
\quad
y_{2}
=
w+6,
\quad
y_{3}
=
w+7.
$$

The latest start and finish times are given by setting $w=0$ in the form
\begin{equation*}
x_{1}
=
0,
\quad
x_{2}
=
4,
\quad
x_{3}
=
3,
\quad
y_{1}
=
3,
\quad
y_{2}
=
6,
\quad
y_{2}
=
7.
\end{equation*}
\end{example}

\section{Conclusions}

In many tropical optimization problems encountered in real-world applications, it is not too difficult to obtain a particular solution in an explicit form, whereas finding all solutions may be a hard problem. This paper was concerned with multidimensional optimization problems that arise in various applications as the problems of minimizing and maximizing the span seminorm. To obtain a complete solution of the minimization problem, we first characterized all solutions by a system of simultaneous vector equation and inequality, and then developed a new matrix sparsification technique. This technique was applied to describe all solutions in an explicit vector form. As another use of sparsified matrices, we derived a compact vector representation for complete solution of the maximization problem. The results obtained were applied to find a complete solution to a real-world scheduling problem.

The extension of the characterization of solutions and sparsification technique proposed in the paper to other tropical optimization problems may present important directions for future work. New applications of the results to solve real-world problems in various fields, including location analysis and decision making, are of particular interest. The connection between tropical optimization and relational algebra can be another line of future research.

\section*{Acknowledgments}
This work was supported in part by the Russian Foundation for Humanities (grant number 16-02-00059). The author is very grateful to three referees for their valuable comments and suggestions, which have been incorporated into the revised version of the manuscript.

\bibliographystyle{utphys}

\bibliography{Algebraic_solution_of_tropical_optimization_problems}

\end{document}